\documentclass[12pt]{amsart}
\usepackage{amsmath, amssymb, amsthm}
\usepackage{comment, empheq}
\usepackage{mathdots, breqn, xcolor}
\usepackage[colorlinks]{hyperref}
\usepackage{mathtools}
\usepackage[capitalize,nameinlink,noabbrev]{cleveref}
\usepackage{dsfont}
\numberwithin{equation}{section}
\newtheorem{theorem}{Theorem}
\newtheorem{lemma}{Lemma}[section]
\newtheorem{prop}[lemma]{Proposition}
\newtheorem{conj}{Conjecture}

\theoremstyle{definition}
\newtheorem{defn}[lemma]{Definition}
\newtheorem{remark}[lemma]{Remark}

\crefname{subsection}{Subsection}{Subsections}

\newcommand{\R}{\mathbb{R}}
\newcommand{\bbA}{\mathbb{A}}

\renewcommand{\H}{\mathbb{H}}
\newcommand{\C}{\mathbb{C}}
\newcommand{\Q}{\mathbb{Q}}
\newcommand{\Z}{\mathbb{Z}}
\renewcommand{\O}{\mathrm{O}}
\newcommand{\N}{\mathbb{N}}

\newcommand{\X}{\mathbb{X}}

\newcommand{\B}{\mathcal{B}}

\newcommand{\g}{\mathfrak{g}}

\renewcommand{\a}{\mathfrak{a}}

\newcommand{\GL}{\mathrm{GL}}
\newcommand{\PGL}{\mathrm{PGL}}
\newcommand{\SL}{\mathrm{SL}}

\newcommand{\SO}{\mathrm{SO}}
\newcommand{\PSO}{\mathrm{PSO}}
\newcommand{\PSL}{\mathrm{PSL}}

\newcommand{\PO}{\mathrm{PO}}
\newcommand{\Ht}{\mathrm{ht}}

\newcommand{\Eis}{\mathrm{Eis}}
\newcommand{\Mat}{\mathrm{Mat}}

\newcommand{\res}{\mathrm{res}}
\newcommand{\diag}{\mathrm{diag}}

\newcommand{\ind}{\mathrm{Ind}}
\newcommand{\supp}{\mathrm{supp}}
\newcommand{\dist}{\mathrm{dist}}

\newcommand{\n}[1]{\left\Vert #1\right\Vert }

\newcommand{\calI}{\mathcal{I}}

\newcommand{\calF}{\mathcal{F}}
\newcommand{\calB}{\mathcal{B}}
\newcommand{\calS}{\mathcal{S}}

\newcommand{\calH}{\mathcal{H}}

\newcommand{\cusp}{\mathrm{cusp}}

\newcommand{\disc}{\mathrm{disc}}
\newcommand{\One}{\mathds{1}}
\newcommand{\vol}{\mathrm{vol}}
\renewcommand{\d}{\,\mathrm{d}}

\usepackage{svg}
\usepackage[color,notref,notcite,final]{showkeys}
\renewcommand*{\showkeyslabelformat}[1]{%
   \fbox{\vbox{\hsize=1.1cm\normalfont\small\url{#1}\par}}}

\title{Optimal Diophantine Exponents for $\mathrm{SL}(n)$}

\author{Subhajit Jana}
\address{Queen Mary University of London, Mile End Road, London E14 NS, UK.}
\email{s.jana@qmul.ac.uk}

\author{Amitay Kamber}
\address{Centre for Mathematical Sciences, Wilberforce Road, Cambridge CB3 0WB, UK.}
\email{ak2356@dpmms.cam.ac.uk}

\begin{document}

\begin{abstract}
The \emph{Diophantine exponent} of an action of a group on a homogeneous space, as defined by Ghosh, Gorodnik, and Nevo, quantifies the complexity of approximating the points of the homogeneous space by the points on an orbit of the group. We show that the Diophantine exponent of the $\mathrm{SL}_n(\mathbb{Z}[1/p])$-action on the generalized upper half-space $\mathrm{SL}_n(\mathbb{R})/\mathrm{SO}_n(\mathbb{R})$, lies in $[1,1+O(1/n)]$, substantially improving upon Ghosh--Gorodnik--Nevo's method which gives the above range to be $[1,n-1]$. We also show that the exponent is \emph{optimal}, i.e.\ equals one, under the assumption of \emph{Sarnak's density hypothesis}.
The result, in particular, shows that the optimality of Diophantine exponents can be obtained even when the \emph{temperedness} of the underlying representations, the crucial assumption in Ghosh--Gorodnik--Nevo's work, is not satisfied.
The proof uses the spectral decomposition of the homogeneous space 
and bounds on the local $L^2$-norms of the Eisenstein series.
\end{abstract}

\maketitle

\section{Introduction}

\subsection{Main Theorems}
Let $p$ be a prime. It is a standard fact that $\SL_n(\Z[1/p])$ is dense in $\SL_n(\R)$. Following Ghosh, Gorodnik and Nevo \cite{ghosh2015diophantine,ghosh2018best}, we wish to make this density quantitative. 

Let $\H^n := \SL_n(\R)/\SO_n(\R)$ be the symmetric space of $\SL_n(\R)$. Fix an $\SL_n(\R)$-invariant Riemannian metric $\dist$ on $\H^n$.

We define the \emph{height} of $\gamma\in\SL_n(\Z[1/p])$
as 
\[\Ht(\gamma)=\min\left\{ k \in \N\mid p^{k}\gamma\in \Mat_{n}\left(\Z\right)\right\},\]
where $\Mat_n$ denotes the space of $n\times n$ matrices.

Our work is based on the following definition, motivated by \cite{ghosh2018best} (see \cref{sec:GGN} for comparison).

\begin{defn}\label{def:Diopnatine exponents intro}
The \emph{Diophantine exponent} $\kappa(x,x_0)$ of $x,x_0\in \H^n$ is the infimum over $\zeta<\infty$, such that there exists an $\varepsilon_0 = \varepsilon_0(x,x_0,\zeta)$ with the property that for every $\varepsilon < \varepsilon_0$ there is a $\gamma \in \SL_n(\Z[1/p])$ satisfying
\begin{align*}
\dist(\gamma^{-1} x,x_0)\le \varepsilon & \text{  and  }
\Ht(\gamma)\le \zeta \frac{n+2}{2n} \log_p(\varepsilon^{-1}).
\end{align*}
The \emph{Diophantine exponent of $x_0 \in \H^n$} is
\[
\kappa(x_0)=\inf\left\{\tau\mid \kappa(x,x_0)\le\tau\text{ for almost every }x\in\H^n\right\}.
\]
The \emph{Diophantine exponent of $\H^n$} is
\[
\kappa=\inf\left\{\tau\mid\kappa(x,x_0)\le\tau\text{ for almost every }(x,x_0)\in\H^n\times\H^n\right\}.
\]
\end{defn}

It is shown in \cite{ghosh2018best} (in a more generalized context) that
\begin{prop}
for every $x_0\in \H^n$ and almost every $x\in\H^n$ we have $\kappa(x,x_0)=\kappa(x_0)$, and for almost every $x,x_0\in\H^n$ we have $\kappa(x,x_0)=\kappa(x_0)=\kappa$.
\end{prop}

The artificial insertion of the factor $\frac{n+2}{2n}$ in the definition of $\kappa(x,x_0)$ is to ensure that 
\[\text{for every } x_0\in \H^n \quad\kappa(x_0)\ge 1,\quad\text{and}\quad\kappa\ge 1\] as we explain below. It is thus natural to wonder about a corresponding upper bound for $\kappa(x_0)$, namely whether 
\[\text{for every } x_0\in \H^n \quad\kappa(x_0)= 1,\]
and in particular, whether
\[\kappa = 1.\] In this case we say that $\kappa$ is the \emph{optimal Diophantine exponent}. 

In \cite{ghosh2018best} the Diophantine exponents are studied in great generality for a lattice $\Gamma$ in a group $G$, acting on a homogenous space $G/H$, when $H$ is a subgroup of $G$. In our case $\Gamma = \SL_n(\Z[1/p])$, $G= \SL_n(\R) \times \SL_n(\Q_p)$ and $H= \SO_n(\R) \times \SL_n(\Q_p)$ (see \cref{sec:GGN} for details). The optimality of $\kappa$ is proved in \cite{ghosh2018best} under certain crucial \emph{temperedness} assumptions of the action of $H$ on $L^2 (\Gamma \backslash G)$.
Unfortunately, in our particular situation, the temperedness assumption is not satisfied. As we explain in \cref{sec:GGN}, the arguments of \cite{ghosh2018best} imply the following non-optimal upper bounds on $\kappa$. 

\begin{theorem}[Ghosh--Gorodnik--Nevo, \cite{ghosh2018best} and \cref{sec:GGN}]\label{thm:GGN}
Let $n>1$ be a positive integer.
\begin{enumerate}
    \item For $n=2$, assuming the Generalized Ramanujan Conjecture (GRC) for $\GL(2)$, for every $x_0\in \H^2$ we have $\kappa(x_0)=1$. 
    Unconditionally, for every $x_0\in\H^2$ we have $\kappa(x_0)\le 32/25$.
    \item For $n\ge 3$, for every $x_0\in \H^n$ we have $\kappa(x_0)\le n-1$.
\end{enumerate}
In particular, the same bounds also hold for $\kappa$.
\end{theorem}

\vspace{0.5cm}

One of the goals of this paper is to substantially improve the upper bound of $\kappa$, in particular, to prove that $\kappa$ is \emph{essentially} optimal. Our main theorem is as follows.
\begin{theorem}\label{thm:kappa theorem intro}
Let $n>1$ be a positive integer.
\begin{enumerate}
    \item For $n=2$ or $n=3$ we have $\kappa=1$.
    \item For every $n\ge 4$ we have \[\kappa \le 1+\frac{2\theta_n}{n-1-2\theta_n},\] where $\theta_n$ is the best known bound towards the GRC for $\GL(n)$.
\end{enumerate}
\end{theorem}

\begin{remark}
We refer to \cref{subsec:GRC and density} for the precise definition of $\theta_n$. From \cref{bound-towards-ramanujan} and \cref{eq:luo-rudnick-sarnak} we obtain that
\[\kappa\le\begin{cases}
{11}/{8}&\text{ for }n=4,\\
({n^2+1})/({n^2-n})= 1+O(1/n)&\text{ for }n\ge 5
\end{cases}\]
Notice that the bound on $\kappa$ gets better as $n$ grows. A non-precise reason is that as $n$ grows, the Hecke operator on the cuspidal spectrum gets closer and closer to having square-root cancellation.
\end{remark}

\vspace{0.5cm}

We also show the optimality of $\kappa$ for any $n$ assuming \emph{Sarnak's Density Hypothesis} for $\GL(n)$ as in \cref{conj:density conj alternative} which is a \emph{much weaker} version of the GRC for $\GL(n)$.

\begin{theorem}\label{thm:optimal-SDH}
For every $n$, assuming \emph{Sarnak's Density Hypothesis} for $\GL(n)$ as in \cref{conj:density conj alternative} below,
    we have $\kappa=1$.
\end{theorem}

We consider \cref{thm:optimal-SDH} as a \emph{proof of concept} for the claim that the Diophantine exponent is usually optimal, even without the temperedness assumption. This is in line with \emph{Sarnak's Density Hypothesis} in the theory of automorphic forms, which informally states that the automorphic forms are expected to be tempered \emph{on average} (see discussion in \cref{subsec:GRC and density}). Our result, at least on the assumption of the density hypothesis, also negatively answers a question of Ghosh--Gorodnik--Nevo who asked whether optimal Diophantine exponent implies temperedness; see \cite[Remark~3.6]{ghosh2018best} (\cite{kleinbock2015sphere} and \cite{sardari2019siegel} also provide answers to this question, in different contexts).

\vspace{0.5cm}

Note that \cref{thm:GGN} is about $\kappa(x_0)$ while \cref{thm:kappa theorem intro} and \cref{thm:optimal-SDH} are about $\kappa$. The difference may seem minor but is crucial for the proof. In the general setting of Ghosh--Gorodnik--Nevo, we expect that \emph{usually} $\kappa=1$ (e.g., when $\SL(n)$ is replaced by another group), but $\kappa(x_0)$ may be larger, because of local obstructions. 

An example, based on \cite[Section~2.1]{ghosh2013aut}, is the action of $\SO_{n+1}(\Z[1/p])$ on the sphere $S^{n}$, which we discuss shortly in \cref{subsec:sphere}. In our situation, we conjecture that for every $x_0\in \H^n$ we have $\kappa(x_0) = 1$, but do not know how to prove it even assuming the GRC, except for $n=2$ (as in \cref{thm:GGN}) and $n=3$.
\begin{theorem}\label{thm:local-kappa-bound}
    For $n=3$, assuming the GRC, we have $\kappa(x_0) = 1$ for every $x_0\in \H^n$.
\end{theorem}

\vspace{0.1cm}

\subsection{Almost-covering}

Our proofs of \cref{thm:kappa theorem intro} and \cref{thm:optimal-SDH} use the spectral theory of $L^2(\SL_n(\Z)\backslash\H^n)$. It is therefore helpful to understand the problem in an equivalent language that is more suitable for the spectral theory and is of independent interest.

First, it suffices to assume that $x,x_0\in \X:=\SL_n(\Z) \backslash \H^n$ as the Riemannian distance $\dist$ is left-invariant under $\SL_n(\Z)$ and the height function $\Ht$ is bi-invariant under $\SL_n(\Z)$.
The set of points of the form $\SL_n(\Z)\gamma x_0 \in \X$ for $\gamma\in\SL_n(\Z[1/p])$ with $\Ht(\gamma)\le k$ is in bijection with the set $R(1) \backslash R(p^{kn})$, where $R(p^{kn}) := \{A \in \Mat_n(\Z)\mid \det(A)=p^{kn}\}$. More precisely, there is a bijection $$\SL_n(\Z) \backslash \{\gamma\in\SL_n(\Z[1/p])\mid \Ht(\gamma)\le k\}\cong R(1) \backslash R(p^{kn})$$ given by multiplication by $p^k$, and the bijection above holds for generic $x_0$. In general, we have a surjection from $R(1) \backslash R(p^{kn})$ to the set on the left hand side above. 

It is well known that 
\[|R(1)\backslash R(p^{kn})|\asymp p^{kn(n-1)},\] 
See \cref{subsec:Hecke operators} (and \cref{subsec:notations} for the notation $\asymp$). 

The parameter $\kappa(x_0)$ measures the \emph{almost-covering} of $\X$ by the set of points above. Consider a sequence of natural numbers $k=k(\varepsilon)$, such that the $\varepsilon$-balls around the $|R(1)\backslash R(p^{kn})|$ points in $\X$ of the form $\SL_n(\Z)\gamma x_0$ with $\Ht(\gamma)\le k$ cover all but $o(1)$ of the space $\X$, when $\varepsilon\to 0$ (compare \cite[Proposition~3.1]{parzanchevski2018super}). The number $\kappa(x_0)$ is closely related to the growth of $k(\varepsilon)$ as $\varepsilon\to 0$.

Therefore, it is required that as $\varepsilon\to 0$,
\[
m(B_\varepsilon)|R(1)\backslash R(p^{kn})|\ge m(\X)-o(1),
\]
where $m$ denotes the $\SL_n(\R)$-invariant measure on $\X$ and $m(B_\varepsilon)$ is the volume of a ball of radius $\varepsilon$ in $\H^n$. We have $m(B_\varepsilon)\asymp \varepsilon^d$ where $d:=\dim \H^n=\frac{(n+2)(n-1)}{2}$. Thus, we deduce that
\[
k(\varepsilon)\ge \frac{d}{n(n-1)}\log_p(\varepsilon^{-1})(1-o(1)).
\]
The same argument shows that $\kappa(x_0)\ge 1$.

We remark that one can also consider the problem of \emph{covering}, where we would like to cover an entire compact region of $\X$ by small balls of radius $\varepsilon$ around the Hecke points (unlike \cref{thm:kappa theorem intro} and \cref{thm:optimal-SDH}, which are essentially about \emph{almost-covering}). This is also the difference between part $(i)$ and part $(ii)$ of \cite[Theorem 1.3]{ghosh2013aut}. See also \cite{chiu1995covering} for the covering problem of Hecke points around $e\in \X$. We will not discuss it further in this work but mention that our methods, and in particular the arguments in \cref{thm:local-kappa-bound}, can lead to a better understanding of the covering problem as well, but the results are not expected to be optimal. For example, in the covering problem on the $3$-dimensional sphere the spectral approach leads to a covering exponent which is $3/2$ times the conjectural value (see \cite{browning2019twisted} and the references therein).

\begin{figure}[htpb]
    \centering
    \includegraphics[width=6cm]{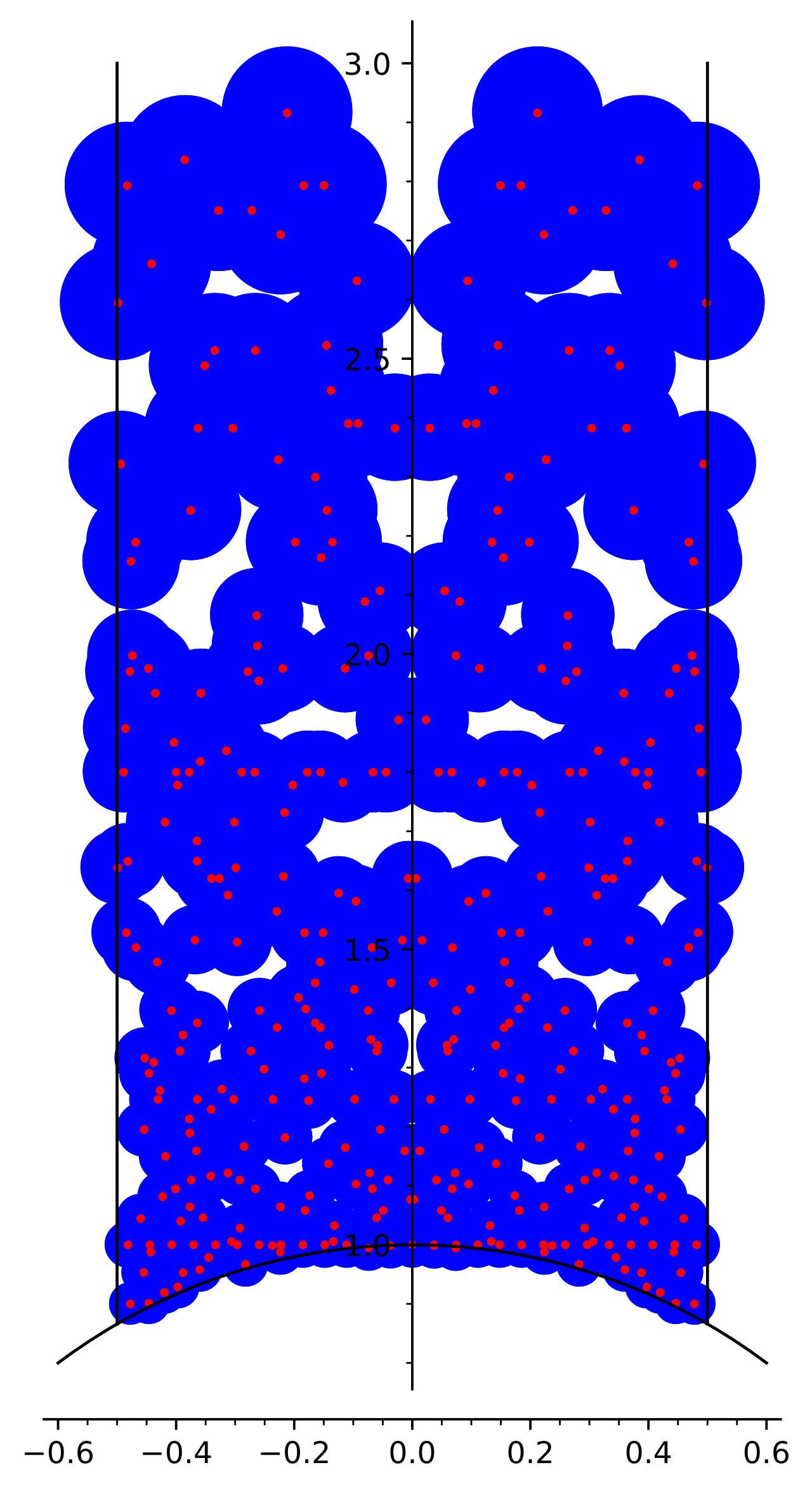}
    \caption{Covering of $\SL_2(\Z)\backslash\H^2$ by balls around $\SL_2(\Z)\backslash\SL_2(\Z[1/3])$. The height of the points is bounded by $3$, and the balls are of radius $3^{-3}$.}
    \label{fig:density}
\end{figure}


\subsection{Outline of the proof}

We start by describing the work of Ghosh--Gorodnik--Nevo in \cite{ghosh2018best} in our setting, namely, on $\X:=\SL_n(\Z)\backslash \SL_n(\R)/\SO_n(\R)$. Their work crucially relies on the existence and, in fact, an explicit description of the \emph{spectral gap} for a certain averaging operator (i.e., \emph{a quantitative mean ergodic theorem}) on a certain homogeneous space arising from a general reductive group. In our case, the situation is simpler and the relevant operator turns out to be the (adjoint) Hecke operator $T^*(p^{nk})$, for certain $k$, acting on $L^2(\X)$. By a standard duality theorem for Hecke operators \cite{clozel2001hecke}, the action reduces to an operator from the Hecke algebra of $\SL_n(\Q_p)$ on $L^2(\SL_n(\Z[1/p])\backslash \SL_n(\R)\times \SL_n(\Q_p))$.

One can relate the spectral gap of this Hecke operator to the Diophantine exponent (see \cref{thm:GGN not intro}). The spectral gap, using the theory of spherical functions, can be determined by a certain \emph{integrability exponent}, which is parameterized by a number $2\le q\le \infty$ (see \cref{prop: GGN spectral gap}). When the integrability exponent is $q=2$ (alternatively, an underlying representation is \emph{tempered}, see \cref{sec:GGN}) one gets the optimal Diophantine exponent $\kappa=1$. In general, the method of \cite{ghosh2018best} only shows that $\kappa \le q/2$ (see \cref{thm:GGN not intro}).

Assuming $n\ge3$, explicit property $(T)$ implies that $q\le 2(n-1)$ (see the work of Oh \cite{oh2002uniform} for a nice proof). This gives the claimed result that $\kappa \le n-1$. On the other hand, the theory of Eisenstein series implies that $q\ge 2(n-1)$ (see \cite[Theorem~1.5]{clozel2001hecke}). As the spectral gap cannot be further improved, the method in \cite{ghosh2018best} is limited and we need a different approach to improve the upper bound of $\kappa$. 

\vspace{0.5cm}

One of the novelties in our work is to use the full spectral decomposition of $L^2(\X)$ and treat different types of elements of the spectrum separately.
More precisely, we actually analyze the spectral decomposition of $L^2(\X_0)$, where $\X_0 := \PGL_n(\Z) \backslash \H^n$. From our experience, this maneuver usually gives more detailed knowledge of the sizes of the \emph{Hecke eigenvalues} than that is provided by the spectral gap alone. The basic spectral decomposition uses the \emph{theory of Eisenstein Series} due to Langlands (see \cite{moeglin1995spectral}) and provides a decomposition of the form $L^2(\X_0)\cong L^2_{\cusp}(\X_0) \oplus L^2_{\Eis}(\X_0)$. The relevant Hecke operator acts on each part of the spectrum and gives rise to certain Hecke eigenvalues.

The cuspidal part $L^2_{\cusp}(\X_0)$ decomposes discretely into irreducible representations. The Generalized Ramanujan Conjecture (GRC) predicts that all such representations are tempered, i.e. the sizes of the Hecke eigenvalues can be bounded optimally. While the GRC is completely open even for $n=2$, good bounds towards it for each individual representation are known; see \cite{sarnak2005notes}. This bound is used in our unconditional results for $n\ge 4$. However, even for $n=2$ and assuming the best-known bounds, we are unable to reach $\kappa=1$.

To overcome this problem, we notice that one does not need optimal bounds for individual Hecke eigenvalues, but \emph{optimal bounds on average}. In general, Sarnak's Density Hypothesis (see \cite{sarnak1991bounds,sarnak1991diophantine}) predicts (in a slightly different setting) that the GRC should hold on average for a nice enough family of automorphic representations. In our conditional result, we assume a certain form of the density hypothesis, namely \cref{conj:density conj} which can be realized as a higher rank analogue of Sarnak's Density Hypothesis (see \cref{lem:density conjecture equivalence} and discussion there), and apply to our question. This approach was already used in different contexts to deduce results of a similar flavor (see \cite{sarnak2015lettermiller,bordenave2021cutoff,golubev2019cutoff,golubev2020sarnak}). The version that is relevant for us can be realized as \emph{Density relative to the Weyl's law}.
We refer to \cref{subsec:GRC and density} for a complete discussion. This density property is known for $n=2$ and $n=3$ by the work of Blomer \cite{blomer2013applications} and Blomer--Buttcane--Raulf \cite{blomer2014sato}, respectively; see \cref{subsec:GRC and density}.
\begin{remark}
Recently, Assing and Blomer in \cite[Theorem 1.1]{assing2022density} proved Sarnak's density hypothesis in a non-archimedean aspect, namely for the automorphic forms for principal congruence subgroups of square-free level. As an application they solved a related problem in \cite[Theorem 1.5]{assing2022density} namely, \emph{optimal lifting} for $\SL_n(\Z/q\Z)$ with square-free $q$, conditional on a hypothesis \cite[Hypothesis 1]{assing2022density} about certain local $L^2$-bounds of the Eisenstein series; see also \cite[Theorem 4, \S8]{jana2022eisenstein-average}.
\end{remark}

Dealing with the Eisenstein part $L^2_{\Eis}(\X_0)$ is less complicated arithmetically than the cuspidal spectrum -- the size of the Hecke eigenvalues can be understood inductively using the results of M{\oe}glin and Waldspurger \cite{moeglin1989residual}. However, the Eisenstein part is more complicated analytically, because of its growth near the cusp. 
This problem can be regarded as a Hecke eigenvalue weighted Weyl's law. Similar to the proofs in \cite{miller2001existence,muller2007weyl} we need to show that the contribution of the Eisenstein part is small compared to the cuspidal spectrum. 
The problem is non-trivial due to the weights coming from the Hecke eigenvalues, which can be quite large for the non-tempered part of the Eisenstein spectrum. We show that the largeness of the Hecke eigenvalues for the non-tempered automorphic forms is compensated by a \emph{low cardinality} of such forms.

The exact result that we need is estimates on the \emph{$L^2$-growth of Eisenstein series in compact domains}, see \cref{subsec:L2 bounds on Eisenstein} for a formulation. This result was proved by Miller in \cite{miller2001existence} as the main estimate in his proof of Weyl's law for $\SL_3(\Z)$, but was open for $n\ge 4$. In a companion paper \cite{jana2022eisenstein-average} we solve this problem; see \cref{subsec:L2 bounds on Eisenstein}.

\subsection{Generalizations and Open Problems}
The questions in this work can be generalized to other groups, by replacing the underlying group $\SL(n)$ by another semisimple simply connected algebraic group. Without giving the full definitions, we expect that the Diophantine exponent $\kappa$ will be optimal, even without the presence of an optimal spectral gap (compare the discussion in \cite[Remark~3.6]{ghosh2018best}).
Our proof certainly generalizes to $\SL_m(D)$, when $D$ is a division algebra over $\Q$. 

One can also wonder about the Diophantine exponents when $\H^n$ is replaced with $\SL_n(\R)$, with some left-invariant Riemannian metric. The methods of this paper can, in principle, be used for this problem as well, but the spectral decomposition, as well as other analytical problems, are more complicated due to the absence of sphericality, and we were not yet able to overcome them. However, one can show that in this situation $\kappa(x_0)=\kappa$, since we have a right $\SL_n(\R)$-action (the metric is not right-invariant, so this is not completely trivial).

\subsection{Structure of the article}
In \cref{sec:GGN} we explain how our question is related to the work of Ghosh--Gorodnik--Nevo, and prove \cref{thm:GGN}.

In \cref{sec:local preliminaries} we discuss the relevant local groups and their (spherical) representation theory.

In \cref{sec:global preliminaries} we discuss the global preliminaries that we need, and in particular discuss Hecke operators, Langlands spectral decomposition, and the description of the spectrum. In \cref{subsec:GRC and density} we discuss the Density Conjecture that we require for this work, and in \cref{subsec:L2 bounds on Eisenstein} we discuss $L^2$-bounds on Eisenstein series in compact domains.

In \cref{sec:reduction to spectral} we reduce the study of Diophantine exponents to a certain analytic problem (cf. \cref{lem:kappa from L2 global convex}), in the spirit of \cite{ghosh2018best}. 

In \cref{sec:applying spectral decomposition} we apply Langlands spectral decomposition and \cref{thm:maass-selberg} to reduce the spectral problem to a combinatorial problem.  

In \cref{sec:proof of main theorem} we complete the proof of \cref{thm:kappa theorem intro} and \cref{thm:optimal-SDH}.

Finally, in \cref{sec: local exponents n=3} we prove \cref{thm:local-kappa-bound}.

\subsection{Notation}\label{subsec:notations}

The notations $\ll,\asymp,\gg, O,o$ are the usual ones in analytic number theory: for a master parameter $T\to\infty$ and $A,B$ depending on $T$ implicitly we say $A\ll B$ (equivalently, $A=O(B)$, and $B\gg A$) if there is a constant $c$ such that $A\le cB$ for $T$ sufficiently large. We write $A\asymp B$ if $A\ll B\ll A$. The implied constants may depend on $n$ and $p$, without mentioning it explicitly. Also as usual in analytic number theory, $\delta$ and $\eta$ (but not $\varepsilon$) will denote arbitrary small but fixed positive numbers, whose actual values may change from line to line.

\subsection{Acknowledgements}
The authors thank Amos Nevo for various discussions and detailed comments about a previous version of this work and thank Peter Sarnak for his motivation. The first author thanks Shreyasi Datta for numerous fruitful discussions on various aspects of Diophantine approximation related to this work, and in general. The first author also thanks MPIM Bonn where most of the work was completed during his stay there. The second author is supported by the ERC under the European Union's Horizon 2020 research and innovation program (grant agreement No. 803711). Finally, we thank the anonymous referee for the helpful comments which made the exposition of the paper a lot better.

\section{The Setting of Ghosh--Gorodnik--Nevo}\label{sec:GGN}

The goal of this section is to explain how our question fits into the general framework studied by Ghosh--Gorodnik--Nevo in a sequence of works \cite{ghosh2013aut,ghosh2014metric,ghosh2015diophantine,ghosh2018best}. Consequently, we give a sketch of the argument which leads to a proof of \cref{thm:GGN}.
In this section we follow \cite[Section~2]{ghosh2018best} and use its notations, which are not the same as the rest of this work, to allow a simple comparison.

Let $G := \SL_n(\R) \times \SL_n(\Q_p)$ and $H := \SO_n(\R) \times \SL_n(\Q_p) < G$. Also let $\Gamma := \SL_n(\Z[1/p])$, which we consider as embedded diagonally in $G$. Notice that $\Gamma$ is a lattice in $G$.

Let $X:= G/H\cong \SL_n(\R)/\SO_n(\R) \cong \H^n$. Also let $\dist$ be the natural Riemannian metric on $X$, coming from the Killing form on the Lie algebra of $\SL_n(\R)$. Notice that the action of $G$ on $G/H$ preserves this metric. We fix natural Haar measures $m_G,m_H,m_X=m_{G/H}$ on $G,H,X$, respectively.

We define $D\colon \SL_n(\Q_p)\to \R_{\ge 0}$ by 
\[D(g_p) = \log(p) \Ht(g_p) = \log(p) \min\left\{k\in \Z_{\ge 0}\mid p^k g_p \in \Mat_n(\Z_p)\right\}.\]
We extend $D\colon G \to \R_{\ge 0}$ by $D(g_\infty,g_p) := D(g_p)$ and denote $|g|_D := e^{D(g)}$, which in our case is simply the $p$-adic valuation of $g_p$.

\begin{defn}[\cite{ghosh2018best}, Definition~2.1]\label{def: diophantine exponenets GGN}
Given $x,x_0 \in X$, the \emph{Diophantine exponent} $\kappa_1(x,x_0)$ is the infimum over $\zeta$, such that there is an $\varepsilon_0=\varepsilon(x,s_0,\zeta)$ with the property that for every $\varepsilon<\varepsilon_0$ there is a $\gamma\in \Gamma$ satisfying 
\begin{align*}
    \dist( \gamma^{-1}x,x_0) \le \varepsilon\text{    and    }
    |\gamma|_D\le \varepsilon^{-\zeta}.
\end{align*}
\end{defn}

Let us compare \cref{def: diophantine exponenets GGN} and \cref{def:Diopnatine exponents intro}. The inequality
$|\gamma|_D = p^{\Ht(\gamma)}\le \varepsilon^{-\zeta}$
is equivalent to the inequality
$\Ht(\gamma) \le \zeta \log_p(\varepsilon^{-1})$.
Therefore, 
\[\kappa(x,x_0) = \frac{2n}{n+2}\kappa_1(x,x_0).\] 

We claim that the choices of $G$, $H$, $\dist$, and $D$ satisfy the Assumptions 1 - 4 of \cite[Section~2]{ghosh2018best}. First, $\dist$ is $G$-invariant so \cite[Assmption~1]{ghosh2018best} holds (in fact, \cite{ghosh2018best} requires a weaker ``coarse metric regularity'' condition). It also holds that for $\varepsilon$ sufficiently small, 
\[m_{X}(\{x\in X\mid \dist(x,x_0)<\varepsilon\}) \asymp \varepsilon^{d},\]
where $d:=\dim(X)= \dim(\SL_n(\R)/\SO_n(\R)) = (n+2)(n-1)/2$. Therefore, \cite[Assmption~4]{ghosh2018best} holds with local dimension $d$.
Next, it directly follows from the definition that $D(g_1g_2)\le D(g_1)+D(g_2)$, which means that $D$ is subadditive. While \cite[Assmption~2]{ghosh2018best} requires a weaker ``coarse norm regularity" condition for $D$ to hold. 

Given $t\ge0$, the set $H_t = \{h\in H \mid D(h) \le t\}$ is of finite Haar measure and moreover, a simple calculation (see \cref{subsec:bounds on spherical functions}) shows that 
\[m_H(H_t) \asymp e^{n(n-1)t}.\]
Therefore, \cite[Assmption~3]{ghosh2018best} is satisfied with an explicit exponent $a:= n(n -1)$.
\begin{remark}
Our analysis is slightly different than \cite{ghosh2018best}, since the set $G_t = \{g\in G \mid D(g) \le t\}$ is not of finite Haar measure, as assumed in \cite{ghosh2018best}. There are two ways to overcome this minor difference. The first is to ignore it, since this assumption is not used in the proof of \cite{ghosh2018best}. Alternatively, One may define a different metric by 
\[
D_1(g_\infty,g_p):= D(g_p)+\log\|g_\infty\|,
\]
where $\|\cdot \|$ is some submultiplicative matrix norm on $M_n(\R)$.

Then it is not hard to show changing $D$ to $D_1$ will give the same Diophantine exponent, and $G_{t,1} = \{g\in G \mid D_1(g) \le t\}$ will be of finite Haar measure. However, the relation between $D_1$ and Hecke points is less transparent, so we prefer to work with $D$ instead.
\end{remark}

In \cite{ghosh2018best} the authors made the following observation: $\kappa(x,x_0)$ is a $\Gamma\times \Gamma$-invariant function, and $\Gamma$ acts ergodically on $X$, and therefore for every $x_0\in X$ there are constants $\kappa_L(x_0)$ and $\kappa_R(x_0)$ such that $\kappa_L(x_0)=\kappa(x_0,x)$ and $\kappa_R(x_0)=\kappa(x,x_0)$ for almost every $x$. Similarly, there is a constant $\kappa$ such that $\kappa= \kappa(x,x_0)$ for almost every $x,x_0\in X$. 

We have the following lower bound of the Diophantine exponent $\kappa_1(x,x_0)$.
\begin{prop}[\cite{ghosh2018best}, Theorem~3.1]
For every $x_0\in X$ and for almost every $x\in X$, $\kappa_1(x,x_0)\ge d/a = \frac{n+2}{2n}$, or alternatively $\kappa(x,x_0)\ge 1$.

Therefore, for every $x_0\in X$, $\kappa_R(x_0)\ge 1$ and $\kappa\ge 1$. 
\end{prop}

We gave a sketch of the proof of this proposition in the introduction. We remark that \cite[Theorem~3.1]{ghosh2018best} actually show that for every $x_0\in X$ it holds that $\kappa_L(x_0)\ge 1$, but the above statement follows either by the same arguments, or by replacing $D$ with $D'$, where $D'(g)=D(g^{-1})$.

To present upper bounds, we consider the right action of $H$ on $Y:=\Gamma \backslash G$, and we endow $Y$ with the natural finite Haar measure coming from $G$. 
Let $\beta_t\in C_c(H)$, $\beta_t = \frac{\One_{H_t}}{m_H(H_t)}$ be the normalized characteristic function of $H_t$. We consider the operator of $\pi_Y(\beta_t)$ on $L^2(Y)$, defined by 
\[
(\pi_Y(\beta_t)f) (y):= \frac{1}{m_H(H_t)}\intop_{H_t} f(yh) \d m_H(h).
\]
In this case, $\pi_Y(\beta_t)$ can, in fact, be interpreted as a certain spherical Hecke operator for which we have the following mean ergodic theorem. 

For any measurable space $X$ we denote 
\[L^2_0(X) := \left\lbrace f \in L^2(X)\mid \intop_X f(x) \d x = 0\right\rbrace.\]

\begin{prop}[\cite{ghosh2013aut}, Theorem~4.2]\label{prop: GGN spectral gap}
There is an explicit $q=q(n)>0$ such that as an operator on $L^2_0(Y)$
\[
\|\pi_Y(\beta_t)\|_{\mathrm{op}} \ll_\delta m_H(H_t)^{-q^{-1}+\delta},
\]
for every $\delta>0$.
\end{prop}

The value $q$ in the above proposition is the \emph{integrability exponent} for the action of $H$ on $L^2_0(Y)$. The integrability exponent $q$ is the infimum over $q'$ such that the $K_H$-finite matrix coefficients are in $L^{q'}(H)$, where $K_H$ is a maximal compact subgroup of $H$.
We have the following results on the integrability exponent.
\begin{enumerate}
    \item For $n=2$, using Kim--Sarnak bound towards the Generalized Ramanujan Conjecture (see \cite{sarnak2005notes}), one can take $q = 64/25$. Assuming GRC, we have $q=2$.
    \item For $n\ge 3$, using explicit property $(T)$ from \cite{oh2002uniform} we have $q = 2(n-1)$. Moreover, this choice of $q$ is the best possible.
\end{enumerate}
Using these bounds on the integrability exponents the following result is proved in \cite{ghosh2018best}.

\begin{theorem}[\cite{ghosh2018best}, Theorem~3.3]\label{thm:GGN not intro}
For every $x_0 \in X$ and almost every $x \in X$,
\[
\kappa_1(x,x_0)\le \frac{qd}{2a}.
\]
Therefore, $\kappa_R(x_0)\le q/2$, and consequently, $\kappa \le q/2$.
\end{theorem}

This recovers \cref{thm:GGN} from the results of \cite{ghosh2018best}.

\subsection{Diophantine Exponents on the sphere}\label{subsec:sphere}
This subsection is independent of the rest of the article, and serves to discuss the difference between $\kappa$ and $\kappa(x_0)$.

The possible difference between $\kappa$ and $\kappa(x_0)$ has similar origins as the failure of temperedness, and also the failure of optimal $L^\infty$-bounds -- embedding of a homogeneous orbit of a subgroup in the space. For example, when $x_0=I$, there is a homogeneous orbit of $\SL_{n-1}(\Z)\backslash \SL_{n-1}(\R) \subset \SL_n(\Z) \backslash \SL_n(\R)$, and many points of the Hecke orbit of $\SL_n(\Z[1/p])$ around $I$ belong to the image of this homogeneous orbit in $\X$. It seems that this concentration is not dramatic enough to change $\kappa(I)$, but for other groups, this may happen. The goal of this subsection is to give an example with $\SL(n)$ replaced by $\SO(n)$.


Let $n\ge 5$ and $\SO(n)$ be the algebraic group which is the stabilizer of the quadratic from $Q(x_1,\dots ,x_{n}):= x_1^2+\dots +x_{n}^2$. 

Replacing $\SL(n)$ by $\SO(n)$, one can study the equidistribution of $\SO_{n+1}(\Z[1/p])$ in $\SO_{n+1}(\R)$. This will help up explain the difference between $\kappa(x_0)$ and $\kappa$, hinted at in the introduction. For technical reasons, we restrict to $p=1\mod 4$. 

We let $G:=\SO_{n+1}(\R)\times \SO_{n+1}(\Q_p)$, $H = \SO_{n}(\R)\times \SO_{n+1}(\Q_p)$, and $\Gamma=\SO_{n+1}(\Z[1/p])$ a lattice in $G$. It holds that $X:=G/H\cong S^n$ and we let $\dist$ be an $\SO_{n+1}(\R)$-invariant Riemannian metric on $X$. We define $D$ and the Diophantine exponent $\kappa_1(x,x_0)$ as above. The relevant dimension is $d=n$, and the set $H_t = \{h\in H\mid D(h)\le t\}$ satisfies 
\[
m_H(H_t) \asymp e^{at},
\]
where
\[ a= 
\begin{cases}
n^2/4 & n\text{ even} \\
(n+1)(n+3)/4 & n\text{ odd}
\end{cases}.
\]
We deduce that for every $x_0\in X$ and almost every $x\in X$ it holds that $\kappa_1(x,x_0)\ge d/a$. The arguments of \cite{ghosh2018best} imply that for every $x_0\in X$ and almost every $x\in X$, it holds that $\kappa_1(x,x_0)\le \frac{qd}{2a}$, where
\[ q= 
\begin{cases}
n & n\text{ even} \\
n+1 & n\text{ odd}
\end{cases}.
\]

We \emph{conjecture} that for almost every $x,x_0\in X$ it holds that \[\kappa_1(x,x_0)=d/a.\] We plan to pursue this conjecture in a future work, using the methods of this work.

Now consider the point $x_0=e=(1,\dots ,0)\in X$. For this specific point, the cardinality of the set of points of the form $\gamma e$ with $D(\gamma)\le t$ is at most the number of solutions to 
\[x_1^2+\dots +x_{n+1}^2 = p^k,\] 
with $x_i \in \Z$ and $p^k \le e^{2t}$. 
It is standard that the last number is 
\[ \ll_\varepsilon e^{t(n-1+\varepsilon)}.\]

Therefore, by the same arguments, for almost every $x\in X$ it holds that 
\[
\kappa_1(x,e)\ge d/(n-1).
\]
Notice that this is a lot larger than $d/a$.

For $n$ odd, Sardari \cite[Corollary~1.7]{sardari2019siegel} indeed proved that, for almost every $x \in X$ it holds that 
\[
\kappa_1(x,e)= d/(n-1).
\]
The proof uses deep results from automorphic forms to show that the mean ergodic theorem has actually a better spectral gap than given simply by explicit property $(T)$. 

The reader is also referred to \cite{kleinbock2015sphere} for calculation of the Diophantine exponents of the $\SO_{n+1}(\Q)$-action on the sphere, which is proved by a different method, not directly related to the spectral decomposition.

\section{Preliminaries - Local Theory}
\label{sec:local preliminaries}

In this section we describe some results about spherical representations and the spherical transform of $\SL_n(\R)$ and $\SL_n(\Q_p)$. We mainly follow \cite[Section~3]{Duistermaat1979spectra} and \cite[Section~3]{lindenstrauss2007existence} (see also \cite[Section~3]{ghosh2013aut}).

\subsection{Basic Set-up}

For any ring $R$ the group $\GL_n(R)$ denotes the group defined by the invertible elements of the $n\times n$ matrix algebra over $R$, which we call $\Mat_n(R)$. Let $\PGL_n(R)$ be the group $\GL_n(R)/R^\times$.
We have a map of algebraic groups $\GL_n\to\PGL_n$.

We let $v = \infty$ or $v =p$ a prime, and let $\Q_v$ be the corresponding local field, i.e., $\Q_\infty=\R$ or the $p$-adic field $\Q_p$. Let $|\cdot|_v$ be the usual valuation, i.e. $|x|_\infty = |x|$ for $x\in\R$, and $|p^l z|_p=p^{-l}$ for $z\in \Z_p^\times$.

Let $G = G_v := \PGL_n(\Q_v)$. If it is clear from the context we will drop $v$ from the notation. Let $P$ be the subgroup of upper triangular matrices, $N$ be the subgroup of upper triangular unipotent matrices, and $A$ be the subgroup of diagonal matrices. We have $P=NA=AN$. Let $K$ be the standard maximal compact subgroup of $G$, i.e., $K=K_\infty:=\PO_n(\R)$ when $v=\infty$ and $K=K_p:=\PGL_n(\Z_p)$ when $v = p$.
We have the Iwasawa decomposition $G = PK$. When $v=\infty$ denote $\dim(G/K)$ by $d$ whose value is $d= \frac{(n-1)(n+2)}{2}$.

We normalize the Haar measure $m=m_G$ on $G$ as in \cite{lindenstrauss2007existence}. In particular, we give $K$ Haar measure $1$. If $v=p$ this normalization uniquely defines the Haar measure on $G$. 
If $v=\infty$, the Killing form on the Lie algebra $\g$ of $G$ defines an inner product on the tangent space of $G/K$, and defines a metric and measure on $G/K$. This uniquely defines the Haar measure on $G$.

Let $A^+ \subset A$ be the set consisting of the projection to $\PGL(n)$ of the elements of the following form:
\begin{itemize}
    \item When $v = \infty$,  $A^+:=\{\diag(a_1,\dots ,a_n)\mid a_1\ge \dots\ge a_n>0\}$.
    \item When $v = p$, $A^+ := \{\diag(p^{l_1},\dots ,p^{l_n})\mid l_1  \le \dots \le l_n\}$.
\end{itemize} 
We have the Cartan decomposition $G = K A^+ K$. 

We let $\a := \{x=(x_1,\dots ,x_n)\in \R^n\mid \sum x_i =0\}$ be the coroot space of $\PGL(n)$. There is a natural map $A\to \a$, given by $\diag(a_1,\dots ,a_n)\mapsto (\log(|a_1|_v),\dots ,\log(|a_n|_v))$ and further normalized to have sum $0$.

Notice that $\a$ is the same space for all $v$. 
We give $\a$ an inner product using the case $v = \infty$. We identify $\a \le \g$ as the Lie algebra of the connected component of the identity in $A$ and let the inner product on $\a$ be the restriction of the Killing form to it. Thus when $v=\infty$ the set $B_b := K\{\exp \alpha\mid\alpha\in \a, \n{\alpha}\le b\}K$ is the ball of radius $b$ in $G / K$ around the identity.
The inner product allows us to identify $\a$ with its dual $\a^*$. We let $\a_{\C}^* = \a^*\otimes_\R \C$, with the natural extension of the inner product.   

For every  $\mu=(\mu_1,\dots ,\mu_n)\in \a_\C^*$, we associate a character $\chi_\mu \colon P\to\C$ by 
\[\chi_\mu(na)=\chi_\mu(a):=\prod_{i=1}^n |a_i|_{v}^{\mu_i},\]
for $a=\diag(a_1,\dots ,a_n)\in A$ and $n\in N$.
Let 
\[\rho := ((n-1)/2,(n-3)/2,\dots ,-(n-1)/2) \in \a_\C^*\] be the half sum of the positive roots. It holds that $\Delta = \chi_{2\rho}$ is the modular character of $P$.

We denote the \emph{Weyl group} of $G$ by $W$ which is isomorphic to the permutation group of $S_n$. 

\subsection{Spherical Transform}
Given $g \in G$, we let $a(g)$ be its $A$ part according to the Iwasawa decomposition $G = NAK $. 
We define the \emph{spherical function} $\eta_\mu\colon G \to \C$ corresponding to $\mu \in \a_\C^*$ by 
\[
\eta_\mu(g) := \intop_{K}\chi_{\mu+\rho}(a(kg))\d k.
\]
Thus $\eta_\mu$ is bi-$K$-invariant.
Moreover, it can be checked, via a change of variable, that $\eta_\mu$ is invariant under the action of $W$ on $\mu$. We can therefore, without loss of generality, assume that for any $\eta_\mu$ the parameter $\mu$ is dominant, i.e., $\Re(\mu_1)\ge\dots\ge\Re(\mu_n)$. In the $p$-adic case we may also assume that $0\le \Im (\mu_i) < 2\pi /\log(p)$.

The \emph{spherical Hecke algebra} of $G$ is the convolution algebra on $C_c^\infty(K\backslash G / K)$, i.e., the convolution algebra of bi-$K$-invariant compactly supported smooth functions on $G$. For $h \in C_c^\infty(K\backslash G / K)$, we let $\tilde{h}\colon \a_\C^* \to \C$ be the \emph{spherical transform} of $h$, defined by\footnote{The normalization here is from \cite{Duistermaat1979spectra}, which is different from \cite{lindenstrauss2007existence}.}
\[
\tilde{h}(\mu) := \intop_{G} h(g) \eta_\mu(g) \d g.
\]
We have the \emph{spherical Plancherel formula} which states that for $h\in C_c^\infty(K\backslash G / K)$,
\[
\intop_G |h(g)|^2 \d g = \intop_{i\a^*} |\tilde{h}(\mu)|^2 d(\mu) \d\mu,
\]
and \emph{spherical inversion formula} which states that
\begin{equation}\label{eq:spherical inversion formula}
h(g) = \intop_{i\a^*} \tilde{h}(\mu) \overline{\eta_{\mu}(g)} d(\mu) \d\mu.
\end{equation}
Here $d(\mu)$ is a smooth function closely related to the \emph{Harish-Chandra's $c$-function}.
For $v= \infty$, we will need the following estimate (see \cite[Equation~3.4]{lindenstrauss2007existence}) 
\begin{equation}\label{eq:estimate of c function}
d(\mu) \ll (1+\|\mu\|)^{d-(n-1)}=(1+\|\mu\|)^{n(n-1)/2}.   
\end{equation}

\subsection{Spherical Representations}
\label{subsec:spherical representations}

We call an irreducible admissible representation $\pi$ of $G$ \emph{spherical} if $\pi$ has a non-zero $K$-invariant vector. It is well known that such a vector is unique up to multiplication by scalar.

We can construct all admissible irreducible spherical representations of $G$ from the unitarily induced principal series representations. Let $\mu\in\a_\C^*$ and $\ind_{P}^{G}\chi_\mu$ denotes the normalized parabolic induction of $\chi_\mu$ from $P$ to $G$. It is an admissible representation and has a unique irreducible spherical subquotient. Conversely, for any irreducible admissible spherical representation $\pi$ we can find a $\mu_\pi\in\a_\C^*$ such that $\pi$ appears as a unique irreducible subquotient of $\ind_{P}^{G}\chi_{\mu_\pi}$. In this case, we call $\mu_\pi$ to be the \emph{Langlands parameter} of $\pi$; see \cite{sarnak2005notes}. 

Let $\pi$ also be unitary. 
In this case, let $v\in\pi$ be a unit $K$-invariant vector. Then it follows from the definition of the spherical function that the corresponding matrix coefficient $\langle \pi(g) v , v \rangle$ is equal to $\eta_{\mu_\pi}(g)$. 
If $h \in C_c(K \backslash G/K)$ then it holds that
\[\pi(h)v = 
\intop_G h(g) \pi(g) v \d g = \tilde{h}(\mu)v.
\]
This follows from the fact that $\pi(h)v$ is $K$-invariant and therefore a scalar times $v$. This scalar may be calculated by evaluating $\langle \pi(h)v, v\rangle$. If $\mu$ is Langlands parameter of some irreducible, spherical and unitary representation, in particular if $\mu\in i\a^*$, then clearly we have $|\eta_\mu(g)|\le 1$.

Let $Q$ be a standard parabolic subgroup of $G$ attached to the partition $n=n_1+\dots+n_r$. The Levi subgroup $M_Q$ of $Q$ is isomorphic to $\GL(n_1)\times\dots\times\GL(n_r)$ modulo $\GL(1)$. We let $\a_Q^*\cong \{(\lambda_1,\dots ,\lambda_r)\in \R^r\mid\sum_{i=1}^r\lambda_i n_i=0\}$ which is embedded in $\a^*\subset\R^n$, as \[(\lambda_1,\dots ,\lambda_r)\mapsto (\lambda_1,\dots ,\lambda_1,\dots ,\lambda_r,\dots ,\lambda_r),\]
where $\lambda_i$ repeats $n_i$ times.
Similarly, we have $\a_{Q,\C}^*=\a_{Q}^*\otimes \C$ embedded in $\a_{\C}^*$. Given $\lambda=(\lambda_1,\dots ,\lambda_r)\in \a_{P,\C}^*$, it defines a character $\chi_{\lambda}$ of $M_Q$ by \[\chi_{\lambda}(\diag(g_1,\dots ,g_r)) = \prod |\det(g_i)|^{\lambda_i},\quad g_i\in\GL(n_i)\]
Given a spherical representation $\pi$ of $M_Q$ of we construct the representation $\pi_\lambda = \pi\otimes\chi_\lambda$.
We realize $\pi_\lambda$ as an representation of $Q$ by tensoring with the trivial representation of the unipotent radical of $Q$. We denote $\ind_Q^G\pi_\lambda$ to be the normalized parabolic induction. Here the normalization is via the character $\chi_{\rho_Q}$, where $\rho_Q$ is the half sum of the positive roots attached to $Q$. We express the Langlands parameters of $\ind_Q^G\pi_\lambda$ in terms of that of $\pi$, and $\lambda$. 

\begin{lemma}\label{langlands-paramteres-induction}
The Langlands parameters of $\ind_Q^G\pi_\lambda$ are $\mu_\pi+\lambda$.
\end{lemma}

\begin{proof}
Assume first that $Q$ corresponds to an ordered partition $n = n_1 + n_2$. Therefore, the Levi part $M$ of $Q$, modulo its center, is equal to $\PGL_{n_1}(\Q_v)\times \PGL_{n_2}(\Q_v)$. Thus, the spherical representation $\pi$ of $M$ with trivial central character is a tensor product of the representations of $\PGL_{n_1}(\Q_v)$ with Langlands parameters $\mu = (\mu_1,\dots ,\mu_{n_1})$ and $\PGL_{n_2}(\Q_v)$ with Langlands parameters $\mu'= (\mu_1^\prime,\dots ,\mu_{n_2}^\prime)$.

The representation $\ind_{Q}^{G}\pi_\lambda$ has a unique subquotient that is spherical. By the description of the spherical representations above and transitivity of induction, the Langlands parameter of the resulting representation is $\mu'' = (\mu_1,\dots ,\mu_{n_1},\mu_1^\prime,\dots ,\mu_{n_2}^\prime)+\lambda$. The general case follows by an inductive argument.
\end{proof}


\subsection{Bounds on Spherical Functions}
\label{subsec:bounds on spherical functions}

In this subsection we will give uniform bounds on the spherical transform of some spherical functions. We will only need the case when $v=p$ is a prime and assume it for the rest of the subsection.

As the spherical function $\eta_\mu$ is bi-$K$-invariant, the value $\eta_\mu(g)$ depends only on the $A^+$ part of $g$ from the Cartan decomposition. 
We will therefore focus on elements $g\in A^+$, which we will assume to be of the form 
\[
g= \diag(p^{l_1},\dots,p^{l_n}),
\]
with $l_1\le\dots\le l_n$.
As described above, we also assume that $\mu$ is dominant, i.e., $\Re(\mu_1) \ge \dots \ge \Re(\mu_n)$.

We record the following bound from \cite[Lemma~3.3]{ghosh2013aut}.

\begin{lemma}\label{lem:upper bound spherical functions}
Let $g\in A^+$, and $\mu\in \a_{\C}^*$ be dominant. We have 
\[|\eta_\mu(g)| \ll_\delta \chi_{-\rho(1-\delta)+\Re(\mu)}(g),\]
for every $\delta>0$.
\end{lemma}
\begin{proof}
Since our notations are different, we repeat the proof of \cite[Lemma~3.3]{ghosh2013aut}. It holds that 
\begin{align*}
|\eta_\mu(g)| 
&\le \intop_K |\chi_{\mu+\rho}(a(kg))|\d k  = \intop_K \chi_{\Re(\mu)}(a(kg))\chi_\rho(a(kg))\d k.
\end{align*}
Since we assume that $g\in A^+$ and $\mu$ is dominant, from \cite[Proposition~4.4.4(i)]{bruhat1972groupes}, we have 
\[
\chi_{\Re(\mu)}(a(kg)) \le \chi_{\Re(\mu)}(g).
\]
Therefore, 
\[
|\eta_\mu(g)| \le \chi_{\Re(\mu)}(g)\intop_K\chi_\rho(a(kg))\d k = \chi_{\Re(\mu)}(g)\eta_0(g).
\]
Finally, $\eta_0$ is Harish-Chandra's $\Xi$-function, which is bounded for $g\in A^+$ by
\[
\eta_0(g) = \Xi(g) \ll_\delta \chi_{-\rho(1-\delta)}(g),
\]
see, e.g., \cite[4.2.1]{silberger1979harmonic}.
\end{proof}

We will also need an estimate of the measure of double cosets below. This is elementary, a proof can be found in \cite[Lemma 4.1.1]{silberger1979harmonic}.

\begin{lemma}\label{lem: Ball size}
For every $g\in A^+$ we have $m(KgK)\asymp \chi_{2\rho}(g)$.
\end{lemma}

We end this subsection with a discussion of spherical transform of a certain spherical function, which will be needed for latter purposes.

First, we want a measurement of how far are parameters from $i\a^*$. In our context, the relevant parameter is as follows. For dominant $\mu\in\a^*_\C$ parameterizing a unitary representation, we define 
\begin{equation}\label{defn-theta}
\theta(\mu):=\max_i \{|\Re(\mu_i)|\} = \max\{\Re(\mu_1),-\Re(\mu_n)\}.
\end{equation} 

We may assume that $0\le \theta(\mu)\le (n-1)/2$ since it is true for every spherical unitary representation.
Notice that $\theta(\mu)=0$ if and only if $\mu \in i\a^*$. Such Langlands parameters are called \emph{tempered}. 

\begin{remark}
For completeness, we write the relation between $\theta$ and the integrability parameter $q$ from \cref{sec:GGN}. 
By \cite[Lemma~3.2]{ghosh2013aut}, 
given $2\le q < \infty$, the following are equivalent: \begin{itemize}
    \item For every $\varepsilon>0$ it holds that $\eta_\mu\in L^{q+\varepsilon}(G)$.
    \item For every $k=1,\dots,n-1$,
        \[\sum_{i=1}^k \Re(\mu_i) \le (1-2/q) \sum_{i=1}^k \rho_i,\]
        where $\rho_i$ is the $i$-th coordinate of $\rho$\footnote{We remark that \cite[Lemma~3.2]{ghosh2013aut} has a typo which is fixed here.}.
\end{itemize}
Denote the maximal $q$ which satisfies the above equivalent conditions by $q(\mu)$. 
Then, in general, we have 
\[q(\mu)\ge\tilde{q}(\mu):=\frac{2(n-1)}{(n-1)-2\theta(\mu)},\] 
while for $n=2$ or $n=3$ it holds that $\tilde{q}(\mu) = q(\mu)$. 
\end{remark}


Given an integer $l\ge 0$, consider the finite set of tuples $0\le l_1\le\dots\le l_n$ such that $\sum_{i=1}^n l_i =l$. Each such sequence defines a \emph{different} element $\diag(p^{l_1},p^{l_2},\dots,p^{l_n})=\diag(1,p^{l_2-l_1},\dots,p^{l_n-l_1})\in A^+$.  
We define 
\begin{equation}\label{defn-mpl}
    M(p^l):= \bigsqcup_{\substack{0\le l_1\le\dots\le l_n\\\sum_{i=1}^n l_i =l}}K\diag(p^{l_1},\dots,p^{l_n})K=\bigsqcup_{\substack{0\le l_1\le\dots\le l_n\\\sum_{i=1}^n l_i =l}}K\diag(1,p^{l_2-l_1},\dots,p^{l_n-l_1})K.
\end{equation}
By applying \cref{lem: Ball size} we obtain:
\begin{equation}\label{ball-size}
m(K\diag(p^{l_1},\dots ,p^{l_n})K) \asymp p^{\sum_{i=1}^n l_i(n+1-2i)}.
\end{equation}
Summing over all the possible choices of $l_1\le \dots\le l_n$ such that $l_1+\dots +l_n=l$, we deduce that 
\begin{equation}\label{size-mpl}
    m(M(p^l)) \asymp p^{l(n-1)}
\end{equation}
where most of the mass is concentrated on the double coset with $l_1=\dots =l_{n-1}=0$, $l_n=l$.

We also define 
\begin{equation}\label{defn-hpl}
h_{p^l} := \frac{1}{m(M(p^l))}{\One_{M(p^l)}}\in C_c^\infty(K\backslash G / K)
\end{equation}
which is the normalized characteristic function of $M(p^l)$.
This operator will correspond to the usual \emph{Hecke operator} $T^*(p^l)$ which we will define in \cref{subsec:Hecke operators}.


\begin{lemma}\label{lem:bounds on lambda from theta}
It holds that for $\mu\in \a_{\C}^*$
\[
|\tilde{h}_{p^l}(\mu)|\ll_\delta p^{l(\theta(\mu) - (n-1)/2+\delta)},
\]
for every $\delta>0$. 

Alternatively, if we write \[\lambda_\mu(p^l) := \tilde{h}_{p^l}(\mu)m(M(p^l))p^{-l(n-1)/2},\] then we have
\[
|\lambda_\mu(p^l)|\ll_\delta p^{l(\theta(\mu)+\delta)},
\]
for every $\delta>0$.
\end{lemma}

\begin{proof}
Note that using the $W$-invariance of $\tilde{h}_{p^l}(\mu)$ it suffices to consider $\mu$ to be dominant.

From \cref{defn-mpl} and the definition of $h_{p^l}$ we have
\[
\tilde{h}_{p^l}(\mu) = \frac{1}{m(M(p^l))}\sum_{l_1,\dots ,l_n} m\left(K\diag(p^{l_1},\dots,p^{l_n})K\right)\eta_\mu\left(\diag(p^{l_1},\dots,p^{l_n})\right),
\]
where the sum is over $0 \le l_1\le\dots\le l_n$ with $l_1+\dots +l_n=l$.
We use \cref{lem:upper bound spherical functions}, \cref{size-mpl}, and \cref{ball-size} to bound the above display equation by
\[\ll_\delta \frac{1}{p^{(n-1)l}} \sum_{l_1,\dots ,l_n} \chi_{\rho(1+\delta)+\Re (\mu)}\left(\diag(p^{l_1},\dots,p^{l_n})\right).\]
Each summand above is bounded by
\[p^{l\left(\frac{n-1}{2}+\delta+\theta(\mu)\right)}.\]
Thus we obtain that $\tilde{h}_{p^l}(\mu)$ is bounded by
\[p^{l\left(-\frac{n-1}{2}+\delta+\theta(\mu)\right)}\sum_{l_1,\dots ,l_n}1.\]
Noting that the last sum is bounded by $l^n\ll_\delta p^{l\delta}$ we conclude.
\end{proof}

\begin{remark}\label{rem:trivial-bound-spherical-transform}
If $\mu$ is Langlands' parameter of a unitary representation then trivially we have $|\tilde{h}_{p^l}(\mu)|\le 1$.
\end{remark}

\subsection{The Paley--Wiener Theorem}
\label{subsec:Paley-Weiner}

In this subsection we will assume that $v= \infty$ and discuss the Paley-Wiener theorem for spherical functions. This is a common tool to localize the spectral side of a trace formula (e.g. see proof of Weyl's law in \cite{muller2007weyl}). See \cite[Section~3]{Duistermaat1979spectra} for details of the results in this subsection.

We define the \emph{Abel--Satake transform} (also known as the \emph{Harish-Chandra transform}) to be the map $C_c(K \backslash G / K)\to C_c(A)$ defined by
\[
f\mapsto \calS f : a\mapsto \Delta(a)^{1/2} \intop_{N}f(an)\d n .
\]
Since $\calS f$ is left $K \cap A$-invariant, it is actually a map on $A^0$, the connected component of the identity of $A$. We have the exponent map $\exp\colon\a\to A^0$ , given by \[(\alpha_1,\dots ,\alpha_n)\to \diag(\exp(\alpha_1),\dots ,\exp(\alpha_n)).\] 
This gives an identification of $\a$ with $A^0$. So we may as well consider
\[
\calS f \in C_c(\a)
\]
after pre-composing with $\exp$ map.

It holds that $\calS f$ is $W$-invariant and Gangolli  showed that 
\[\calS\colon C_c^
\infty(K \backslash G / K) \to C_c^\infty(\a)^{W}\]
is an isomorphism of topological algebras (see \cite[3.21]{Duistermaat1979spectra}).
Harish-Chandra showed that, if we denote the Fourier--Laplace transform $C_c(\a)\to C(\a_\C^*)$ by the map
\[h\mapsto \hat{h}:\mu\mapsto \intop_\a h(\alpha) e^{\mu \alpha}\d\alpha,\]
then it holds that 
\[\tilde{h} = \widehat{\calS(h)}.\]
Gangolli (see \cite[eq. 3.22]{Duistermaat1979spectra}) also proved the following important result.

\begin{prop}\label{thm:Gangolli}
Let $h\in C_c^\infty(\a)^{W}$ be such that $\supp (h) \subset \{\alpha \in \a\mid\|\alpha\|\le b\}$, then \[\supp (\calS^{-1} h) \subset B_b:=K \{\exp \alpha\mid\alpha \in \a, \|\alpha\|\le b\}K.\]
\end{prop}
We record the classical Paley--Wiener theorem, which asserts that for $h\in C_c^\infty(\a)^{W}$ with support in $\{\alpha \in \a\mid\|\alpha\|\le b\}$ we have
\begin{equation}\label{paley-wiener-thm}
|\hat{h}(\mu)|\ll_{N,h} \exp(b\|\Re (\mu) \|)(1+\|\mu\|)^{-N}
\end{equation}
for every $N\ge 0$.

The following lemma is standard (compare, e.g., \cite[Subsection~6.2]{Duistermaat1979spectra}). We give a proof because the lemma is crucial in our work.

\begin{lemma}\label{lem:k_0 conditions}
Let $\varepsilon\to 0$. There exists a function $k_\varepsilon \in C_c(K\backslash G/K)$ that satisfies the following properties:
\begin{enumerate}
    \item $k_\varepsilon$ is supported on $B_\varepsilon:= K\{\exp \alpha\mid\alpha\in \a, \n{a}\le \varepsilon \}K$.
    \item We have $\intop_G k_\varepsilon(g)\d g =1$.
    \item We have $\n{k_\varepsilon}_\infty \ll \varepsilon^{-d}$ and $\n{k_\varepsilon}_2 \ll \varepsilon^{-d/2}$.
    \item The spherical transform satisfies for every $\mu \in a_\C^*$ with $\theta(\mu)\le (n-1)/2$,
    \begin{equation*}\label{eq:spectral decay}
     |\tilde{k}_\varepsilon(\mu)| \ll_N (1+ \varepsilon\|\mu\|)^{-N},
    \end{equation*}
    for all $N>0$.
    \item There is a constant $C>0$ such that for $\mu \in a_\C^*$ satisfying $\|\mu \|\le C\varepsilon^{-1}$ it holds that
    $|\tilde{k}_\varepsilon(\mu)| \gg 1.$
\end{enumerate}

\end{lemma}

\begin{proof}
Choose a fixed $h \in C_c^\infty(\a)^{W}$, having the properties that
\begin{itemize}
    \item $h$ is non-negative,
    \item $\hat{h}(0)= 1$,
    \item $\supp(h)\subset \{\alpha \in \a\mid\|\alpha \|\le 1/2\}$. 
\end{itemize}
By the Paley--Wiener theorem, as in \cref{paley-wiener-thm}, an the third property above we have \[|\hat{h}(\mu)|\ll_{h,N}{{\exp(\|\Re(\mu)\|/2)}} (1+\|\mu\|)^{-N}.\]
In addition, by continuity of $h$ and the second property above we can find a constant $C>0$ such that $|\hat{h}(\mu)|\ge 1/2$ for $\|\mu\|\le C$.

We define $h_\varepsilon(\alpha) := \varepsilon^{{-(n-1)}}h(\alpha/\varepsilon)$.
Then we have $\hat{h}_\varepsilon(\mu) = \hat{h}(\varepsilon\mu)$.  

Finally, we define 
\[k_\varepsilon := C_\varepsilon^{-1} \calS^{-1}(h_\varepsilon),\quad\text{where } C_\varepsilon := \intop_G\calS^{-1}(h_\varepsilon)(g) \d g.\] 
Hence, $\tilde{k}_\varepsilon=C_\varepsilon^{-1}\hat{h}_\varepsilon$.

Let us now prove the different properties of $k_\varepsilon$. Property $(2)$ follows from the normalization. 
By \cref{thm:Gangolli}, $k_\varepsilon$ is supported on $B_\varepsilon$, proving property $(1)$.

To estimate $C_\varepsilon$, we first notice that the spherical function ${\eta_{-\rho}}$ is simply the constant function $1$. So we have \[C_\varepsilon=\intop_G\calS^{-1}(h_\varepsilon)(g) \d g = \intop_G\calS^{-1}(h_\varepsilon)(g)\overline{\eta_{-\rho}}(g) \d g = \hat{h}_\varepsilon(-\rho) = \hat{h}(-\varepsilon\rho).\]
As $\hat{h}(0) = 1$, we have $C_\varepsilon\asymp 1$ as $\varepsilon\to 0$.
This implies that 
\[|\tilde{k}_\varepsilon(\mu)| = |C_\varepsilon^{-1} \hat{h}_\varepsilon(\mu)|= |C_\varepsilon^{-1} \hat{h}(\varepsilon \mu)|\ll_N {{\exp(\varepsilon\|\Re(\mu)\|/2)}}(1+\varepsilon\|\mu\|)^{-N},\]
so property $(4)$ holds. 

Similarly, for $\|\mu\|\le C\varepsilon^{-1}$, 
\[|\tilde{k}_\varepsilon(\mu)| = |C_\varepsilon^{-1} \hat{h}(\varepsilon\mu)| \ge C_\varepsilon^{-1}/2 \gg 1,
\]
so property $(5)$ holds.

Finally, using the spherical inversion formula as in \cref{eq:spherical inversion formula}, we have
\[|k_\varepsilon(g)|\le  \intop_{i\a^*} |\tilde{k}_\varepsilon(\mu)| |\eta_\mu(g)| d(\mu) \d \mu.\]
Using the fact that $|\eta_\mu(g)| \le 1$ for every $g\in G$ and $\mu \in i\a^*$ unitary, property $(4)$, and \cref{eq:estimate of c function} we obtain that the above is bounded by
\[\ll_N \intop_{i\a^*} (1+\varepsilon \|\mu\|)^{-N} (1+\|\mu\|)^{d-n+1} \d \mu.\]
The above integral can be estimated by making $N$ large enough as
\[\ll\intop_0^\infty(1+\varepsilon x)^{-N} x^{d-1} \d x \ll \varepsilon^{-d}.\]
This shows that $\|k_\varepsilon\|_\infty \ll \varepsilon^{-d}$. The bound on $\|k_\varepsilon\|_2$ follows from the bound on $\|k_\varepsilon\|_\infty$, property $(1)$, and the fact that $\vol(B_\varepsilon)\asymp \varepsilon^d$.
\end{proof}

\section{Preliminaries - Global Theory}\label{sec:global preliminaries}

\subsection{Adelic formulation}\label{subsec:adelic formulation}
In this section, we describe global preliminaries in adelic language that are needed for the proof. Temporarily in this section, $p$ will denote a generic finite prime of $\Q$.

Let $\bbA := \R \times \prod_{p}^\prime \Q_p$ be the adele ring of $\Q$, where $\prod^\prime$ means that if $x=(x_\infty,\dots,x_p,\dots)\in\bbA$ then $x_p\in\Z_p$ for almost all $p$. 

We recall the notations $K_\infty = \PO_n(\R)$ and $K_p = \PGL_n(\Z_p)$. We denote $K_{\bbA} := K_\infty \times \prod_{p} K_p$ which is a hyper-special maximal compact subgroup of $\PGL_n(\bbA)$. We also denote
\[\X_{\bbA} := \PGL_n(\Q) \backslash \PGL_n(\bbA) / K_{\bbA}.\]

Recall that $\X := \SL_n(\Z) \backslash \H^n \cong \SL_n(\Z) \backslash \SL_n(\R) / \SO_n(\R)$. 

Let $\varphi\colon \SL_n(\R) \to \PGL_n(\R)$ be the natural quotient map. This map defines an action of $\SL_n(\R)$ on $\PGL_n(\R)/K_\infty$, which is easily seen to be transitive, and since $\varphi^{-1}(\PO_n(\R))=\SO_n(\R)$ we can identify $\H^n = \SL_n(\R)/\SO_n(\R) \cong \PGL_n(\R)/ K_\infty$. By considering the left action of $\SL_n(\Z)$ on the two spaces we can identify 
\[
\X \cong \PSL_n(\Z) \backslash \PGL_n(\R) / K_\infty.
\]
Similarly, by considering the transitive right action of $\SL_n(\R)$ on $\PSL_n(\Z) \backslash \PGL_n(\R)$ we may identify $\SL_n(\Z) \backslash \SL_n(\R) \cong \PGL_n(\Z) \backslash \PGL_n(\R)$ and 
\[
\X \cong \PGL_n(\Z) \backslash \PGL_n(\R) / \PSO_n(\R).
\]
It is simpler for us to work with the space 
\[\X_0 := \PGL_n(\Z) \backslash \PGL_n(\R) / K_\infty,\]
which is a quotient space of $\X$ by the group $\PGL_n(\Z)/\PSL_n(\Z)$ of size $2$.

The following \cref{lem: local-global space equivalence} allows us to identify 
$\X_0$ with $\X_{\bbA}$.
The lemma is well known, but essential for this work, so we provide a proof for completeness.

\begin{lemma}\label{lem: local-global space equivalence}
We have
\begin{equation*}
\X_0  \cong \X_\bbA
\end{equation*}
as topological spaces.
\end{lemma}

\begin{proof}
Let $K_f = \prod_p K_p$ embedded naturally in $\GL_n(\bbA)$. It is enough to prove that 
\[ 
\PGL_n(\Z) \backslash \PGL_n(\R) \cong \PGL_n(\Q) \backslash \PGL_n(\bbA) / K_f.
\]
By the fact that $\GL_n$ over $\Q$ has class number $1$ (see \cite[Proposition 8.1]{platonov1994algebraic}) we have
\[ \GL_n(\bbA) = \GL_n(\Q) \GL_n(\R) \prod_p \GL_n(\Z_p).\]
Alternatively, this follows from the fact that 
\begin{equation}\label{eq:local class number 1}
    \GL_n(\R) \times \GL_n(\Q_p) = \GL_n(\Z[1/p])\left(\GL_n(\R) \times \GL_n(\Z_p)\right),
\end{equation}
where $\GL_n(\Z[1/p])$ is embedded diagonally in the left hand side.
Using the case $n=1$ which states
\[
\bbA^\times = \Q^\times \R^\times \prod_p \Z_p^\times,
\]
we get 
\[ \PGL_n(\bbA) = \PGL_n(\Q) (\PGL_n(\R) \times K_f).\]
We deduce that the right action of $\PGL_n(\R)$ on $\PGL_n(\Q) \backslash \PGL_n(\bbA) / K_f$ is onto. The stabilizer of this action is $\PGL_n(\Z)$, so we get the desired homeomorphism.
\end{proof}

\begin{remark}
Alternatively, it holds that 
\begin{equation}\label{eq:GLn vs PGL_n local}
\X_0 \cong \GL_n(\Z)\R^\times \backslash \GL_n(\R) / \O_n(\R)    
\end{equation}
and
\begin{equation}\label{eq:GLn vs PGL_n global}
\X_\bbA \cong \GL_n(\Q)\R^\times \backslash \GL_n(\bbA) / \O_n(\R)\times \prod_p \GL_n(\Z_p).
\end{equation}
\end{remark}

\begin{remark}\label{rem:local-global X}
A similar proof using the action on $\PSL_n(\R) \subset \PGL_n(\R)$ identifies $\X$ with the adelic space $\PGL_n(\Q) \backslash \PGL_n(\bbA) / (\PSO_n(\R)\times K_f)$ (i.e., $K_\infty = \PO_n(\R)$ is replaced with $\PSO_n(\R)$).
\end{remark}


\subsection{Hecke Operators}
\label{subsec:Hecke operators}

We want to consider Hecke operators on $L^2(\X_\bbA)$ (equivalently, on $L^2(\X_0)$) from a representation theoretic point of view. This is standard (e.g., \cite{clozel2001hecke}), but since we work on $\PGL(n)$, which is not simply connected, some modifications are needed.

Let $l\ge0$ and $p$ be any finite prime. Consider the infinite set \[R(p^l) = \{ x \in \Mat_n(\Z)\mid \det(x) = p^l\}.\]

Recall that for every ring $R$ we have a natural projection $\GL_n(R)\to \PGL_n(R)$, and we will write $\tilde{R}(p^l)$ for the projection of $R(p^l)$ to $\PGL_n(\Z[1/p]) \subset \PGL_n(\Q)$.  Notice that $\tilde{R}(p^l)$ is both left and right $\tilde{R}(1)=\PSL_n(\Z)$-invariant. Let $A(p^l)$ be a finite set of representatives for $\tilde{R}(p^l)/\tilde{R}(1)$. It is possible to explicitly describe the set $A(p^l)$ (\emph{e.g.}, see \cite[Lemma~9.3.2]{goldfeld2006aut} for \emph{right} cosets), but we will not need this explicit description.

\begin{lemma}\label{lem:double coset comparison}
It holds that $K_p \tilde{R}(p^l) K_p = M(p^l)$, where 
$M(p^l) \subset \PGL_n(\Q_p)$ is (as in \cref{subsec:bounds on spherical functions}) the disjoint union of the double cosets of the form 
\[
K_p \diag(p^{l_1},\dots ,p^{l_n}) K_p,
\]
with $ 0 \le l_1 \le\dots \le l_n$ and $\sum l_i = l$. 

Moreover, the elements of $A(p^l)$ can be taken as representatives of the left $K_p$-cosets of $M(p^l)$.
\end{lemma}

\begin{proof}
For $ 0 \le l_1 \le\dots\le l_n$ and $\sum l_i =l$ it holds that $\diag(p^{l_1},\dots ,p^{l_n})\in R(p^l)$. By projecting to $\PGL_n(\Q_p)$ it follows that $M(p^l)\subset K_p \tilde{R}(p^l) K_p$. 

We check the other direction. For each element $\gamma  \in R(p^l)$ (i.e., $\gamma \in \Mat_n(\Z)$ with $\det(\gamma)= p^l$) we can find, using the Cartan decomposition in $\GL_n(\Q_p)$, two elements $k_1,k_2\in \GL_n(\Z_p)$ and $a = \diag(p^{l_1},\dots,p^{l_n})$ with $l_1\le\dots\le l_n$ such that $k_1 \gamma k_2= a$. Therefore, $a\in \Mat_n(\Z_p)$, hence $l_1\ge 0$. Moreover, by comparing the determinants we have $\sum l_i = l$. By mapping to $\PGL_n(\Q_p)$ we deduce that $\tilde{R}(p^l)\subset M(p^l)$ and $K_p \tilde{R}(p^l) K_p \subset M(p^l)$. 

Now, the natural embedding $\tilde{R}(p^l)\to M(p^l)$ extends to a natural embedding \[\tilde{R}(p^l)/\tilde{R}(1) \to M(p^l)/K_p.\] We need to show that this map is surjective. By \cref{eq:local class number 1}, the action of $\PGL_n(\Z[1/p])$ on $\PGL_n(\Q_p)/K_p$ is transitive. Therefore, each left coset $M(p^l)/K_p$ has a representative $\gamma \in \PGL_n(\Z[1/p])$. Lifting to $\GL_n(\Q_p)$ and applying the Cartan decomposition, we can write $\gamma = k \diag(p^{l_1},\dots,p^{l_n}) k'$, for some $k,k' \in \GL_n(\Z_p)$ and $0 \le l_1\le\dots\le l_n$ with $\sum l_i = l$. Therefore $\gamma \in \GL_n(\Z[1/p])\cap \GL_n(\Z_p) \subset \Mat_n(\Z)$. Finally, $\det(\gamma) = p^l$ implies that $\gamma \in R(p^l)$, as needed.
\end{proof}

\begin{defn}
Let $l\ge 0$. The \emph{Hecke operator} $T^*(p^l)$ acting on $L^2(\X_0)$ is defined as 
\[\left(T^*(p^l)\varphi\right)(x):=\frac{1}{|A(p^l)|} \sum_{\gamma\in A(p^l)} \varphi(\gamma^{-1} x).\]
\end{defn}
By the facts that $A(p^l)$ are representatives for left  $\PGL_n(\Z)$-cosets of $\PGL_n(\Z) \tilde{R}(p^l)$ which is right and left $\PGL_n(\Z)$-invariant, and $\varphi$ is left $\PGL_n(\Z)$-invariant, the operator $T^*(p^l)$ is well-defined, and does not depend on the choice of $A(p^l)$.

\begin{remark}
In analytic number theory, one usually defines the Hecke operator $\tilde{T}(p^l)$ as \[
\left(\tilde{T}(p^l)\varphi\right)(x)=\frac{1}{p^{l(n-1)/2}}\sum_{\gamma \in \tilde{R}(1)\backslash \tilde{R}(p^l)} \varphi(\gamma x),
\]
see \emph{e.g.}, \cite[\S 9.3.5]{goldfeld2006aut}. If we define $T(p^l) := \frac{p^{l(n-1)/2}}{|\tilde{R}(1)\backslash \tilde{R}(p^l)|}\tilde{T}(p^l)$ 
then $T^*(p^l)$ is indeed the adjoint of $T(p^l)$.
\end{remark}

\begin{remark}
Alternatively, using \cref{eq:GLn vs PGL_n local}, we could have defined a Hecke operator $T^*_\GL(p^l)$ on the space $L^2(\GL_n(\Z)\R^\times \backslash \GL_n(\R) / O_n(\R))$ by 
\[
\left(T^*_\GL(p^l)\varphi\right)(x) = \frac{1}{|R(p^l)/R(1)|}\sum_{\gamma \in R(p^l)/R(1)}\varphi(\gamma^{-1}x)
\]
This definition agrees with the other definition under the equivalence \cref{eq:GLn vs PGL_n local}.
\end{remark}

Using \cref{lem: local-global space equivalence} we may lift $\varphi \in L^2(\X_0)$ to a function $\varphi_\bbA \in L^2(\X_\bbA)$,
by 
\[\varphi_{\bbA}(g_\infty,(e)_p) := \varphi(g_\infty)\]
where $(e)_p:=(1,1,\dots)\in\prod_p\GL_n(\Q_p)$.
We extend it to be left $\PGL_n(\Q)$-invariant and right $K_\bbA$-invariant. Using such identification we can define certain averaging operators on $L^2 (\X_{\bbA})$ using local functions $h_v \in C_c(K_v \backslash \PGL_n(\Q_v) / K_v)$, such as
\[
(R(h_v) \varphi_\bbA)(\dots,x_v,\dots) := \intop_{\PGL_n(\Q_v)}h_v(y)\varphi_\bbA(\dots,x_vy,\dots) \d y.
\]
Recall \cref{defn-hpl}
\[h_{p^l} := \frac{1}{m(M(p^l))}{\One_{M(p^l)}} \in C_c(K_p\backslash \PGL_n(\Q_p) / K_p).\]
We show that the operators $T^*(p^l)$ and $R(h_{p^l})$ are, in fact, classical and adelic versions, respectively, of one another. 

\begin{lemma}\label{global-local-hecke}
It holds that 
$(T^*(p^l)\varphi)_{\bbA} = R(h_{p^l})\varphi_\bbA$.
\end{lemma}

\begin{proof}
From the definitions of the Hecke operator and the adelic lift above, we have
\begin{equation*}
|A(p^l)|(T^*(p^l)\varphi)_{\bbA} (g_\infty,(e)_q) 
= \sum_{\gamma\in A(p^l)} \varphi_{\bbA}(\gamma^{-1} g_\infty,(e)_q)
= \sum_{\gamma\in A(p^l)} \varphi_{\bbA}(g_\infty,(\gamma)_p,(e)_{q\ne p}), 
\end{equation*}
where the second equality follows from the left $\PGL_n(\Q)$-invariance and the right $K_\bbA$-invariance of $\varphi_{\bbA}$. Once again using the right $K_\bbA$-invariance of $\varphi_\bbA$ we can write the above as
\[\intop_{\PGL_n(\Q_p)}\varphi_{\bbA}(g_\infty,(y)_p,(e)_{q\ne p})\sum_{\gamma\in A(p^l)}\One_{\gamma K_p}(y) \d y.\]
According to \cref{lem:double coset comparison},
\[
\sum_{\gamma \in A(p^l)}  \One_{\gamma K_p} = \One_{M(p^l)}.
\]
In addition, it holds that $|A(p^l)| = m(M(p^l))$. 
Therefore, for $g\in \PGL_n(\bbA)$ of the form $g= (g_\infty,(e)_q)$, it holds that 
\begin{align*}
(T^*(p^l)\varphi)_{\bbA} (g) 
&= \frac{1}{m(M(p^l))}
\intop_{y\in \PGL_n(\Q_p)}\varphi_{\bbA}(gy)\One_{M(p^l)}(y) \d y \\
& = (R(h_{p^l}) \varphi_\bbA)(g),
\end{align*}
as needed.
\end{proof}


\begin{remark}\label{rem:hecke on X}
One can similarly define Hecke operators on $\X$, by identifying \[\X = \PGL_n(\Z) \backslash \PGL_n(\R) / \PSO_n(\R).\]
\end{remark}

\subsection{Discrete Spectrum and Weak Weyl Law}\label{subsec:discrete-spectrum}

We will need to use Langlands' spectral decomposition of $L^2(\X_{\bbA})$. In this subsection we describe the discrete part of the spectrum.

The discrete spectrum $L^2_\disc(\PGL_n(\Q) \backslash \PGL_n(\bbA))$ consists of the irreducible representations $\pi$ of $\PGL_n(\bbA)$ occurring discretely in $L^2(\PGL_n(\Q) \backslash \PGL_n(\bbA))$.
By abstract representation theory, we may write
\[
L^2_\disc(\PGL_n(\Q) \backslash \PGL_n(\bbA)) \cong \oplus_\pi V_\pi,
\]
where $V_\pi$ is the $\pi$-isotypic component. 
By the multiplicity one theorem (\cite{shalika1974multiplicity} and \cite{moeglin1989residual}), each representation $\pi$ of $\PGL_n(\bbA)$ appears in the decomposition with multiplicity at most $1$, so $V_\pi$ spans a representation isomorphic to $\pi$. Moreover, $\pi$ is isomorphic to a tensor product of representations $\pi_v$ of $\PGL_n(\Q_v)$ for $v\le\infty$ (see \cite{flath1979decomposition}).

We restrict our attention to spherical representations $\pi$, that is, those having a non-zero $K_{\bbA}$-invariant vector. Since $\pi$ is equivalent to a tensor product of local representations $\pi_v$, and each local representation has a one dimensional $K_v$-invariant subspace, the $K_{\bbA}$-invariant subspace $\pi^{K_\bbA}$ of $\pi$ has dimension $1$. We choose once and for all a unit \[\varphi_\pi \in V_\pi^{K_\bbA}\] for each $\pi$ with $V_\pi \ne \{0\}$.

We denote the set of all such $\varphi_\pi$ by $\calB_{n}$. Then formally we have
\[
L^2_\disc(\X_\bbA) \cong \bigoplus_{\varphi\in \calB_{n}} \C{\varphi}.
\]
The discrete spectrum $\calB_n$ decomposes naturally into two parts, the \emph{cuspidal part} $\calB_{n,\cusp}$ and the \emph{residual part} $\calB_{n,\res}$. 

Recalling from \cref{subsec:spherical representations}, for every place $v$ we may attach the Langlands parameter $\mu_{\varphi,v}= \mu_{\pi_v}$ to $\varphi$.
When $v=\infty$ we denote 
\begin{equation}\label{defn-nu}
\nu_\varphi := \|\mu_{\varphi,\infty}\|.
\end{equation}
By \cite[Equation~3.17]{Duistermaat1979spectra} $\varphi$ is an eigenfunction of the Laplace--Beltrami operator and its eigenvalue is 
\[\|\rho\|^2 - \|\Re( \mu_{\varphi,\infty}) \|^2 + \|\Im (\mu_{\varphi,\infty})\|^2.\]
In particular, for $\nu_\varphi\ge 1$ the Laplace eigenvalue of $\varphi$ is $\asymp \nu_\varphi^2$.

We will also denote
\[
\theta_{\varphi,p} := \theta(\mu_{\varphi,p}),
\]
according to \cref{defn-theta}.

We define
\begin{equation}\label{defn-ft}
    \calF_T := \{\varphi \in \calB_{n,\cusp}\mid \nu_\varphi \le T\},
\end{equation}
We record the statement of \emph{Weak Weyl law} due to Donnelly, which gives an upper bound of the cardinality $\calF_T$ as $T\to\infty$.
\begin{prop}[\cite{donnelly1982cuspidal}]\label{prop:Weyl law upper bound}
We have
$|\calF_T| \ll T^d$ as $T\to\infty$.
\end{prop}

As a matter of fact, using the corresponding lower bounds by M\"uller in \cite{muller2007weyl} and Lindenstrauss--Venkatesh in \cite{lindenstrauss2007existence}, we know that as $T\to \infty$
\begin{equation}\label{eq:full Weyl law}
|\calF_T| = C T^d(1+o(1)),
\end{equation}
for some explicit constant $C$, but we will not need this stronger result.

\subsection{The Generalized Ramanujan Conjecture and Sarnak's Density Hypothesis}
\label{subsec:GRC and density}

Let $\varphi\in\B_n$ be a spherical discrete series. For every $h_v\in C_c(K_v\backslash \PGL_n(\Q_v)/ K_v)$ one has the operator $R(h_v)$, as defined in \cref{subsec:adelic formulation}, and it holds that
\[
R(h_v) \varphi = \tilde{h}_v(\mu_{\varphi,v})\varphi,
\]
where $\tilde{h}_v$ denotes the spherical transform of $h_v$.
%
When $v=p$, combining \cref{global-local-hecke} and \cref{lem:bounds on lambda from theta} we obtain that for every $x\in \X$,
\[
T^*(p^l)\varphi(x) = \lambda_{\varphi}(p^l)p^{-l(n-1)/2}\varphi(x)
\]
such that for all $\delta>0$
\[
|\lambda_\varphi(p^l)|\ll_\delta p^{l(\theta_{\varphi,p}+\delta)}.
\]

Let $\varphi\in \calB_{n,\cusp}$ be a spherical cusp form. The \emph{Generalized Ramanujan Conjecture} (GRC) predicts that the Langlands parameter of $\varphi$ at every place $v$ are tempered, which is equivalent in our notations to
\[
\theta_{\varphi,v}=0,\quad v\le\infty
\]
The GRC at the place $v=p$ implies essentially the sharpest bounds on the Hecke eigenvalues, in the following form. For every $p$ prime, $l\ge 0$, and $\delta>0$,
\[
|\lambda_{\varphi}(p^l)|\ll_\delta p^{l\delta},
\]
as $p^l\to\infty$.

The GRC is out of reach of current technology, even for $n=2$. However, we have various bounds towards it; see \cite{sarnak2005notes} for a detailed discussion.

The bounds of Hecke eigenvalues can be understood in terms of the bounds of $\theta_{\varphi,p}$. For $n=2$, the best bounds are due to Kim--Sarnak \cite{kim2003functoriality}, for $n=3$ and $n=4$, the same are due to Blomer--Brumley \cite[Theorem 1]{blomer2011ramanujan}, and for $n\ge 5$, they are due to Luo--Rudnick--Sarnak \cite{luo1999generalized}. For $\GL(n)$, these bounds are given by 
\begin{equation}\label{bound-towards-ramanujan}
|\theta_{\varphi,p}|\, \le \theta_n
\end{equation}
where
\[\theta_2=\frac{7}{64},\theta_3=\frac{5}{14},\theta_4=\frac{9}{22},\]
and
\begin{equation}\label{eq:luo-rudnick-sarnak}
\theta_n = \frac{1}{2}-\frac{1}{n^2+1},\quad n\ge 5.    
\end{equation}

Our problem requires estimates of the Hecke eigenvalues which is stronger than \cref{eq:luo-rudnick-sarnak}. On the other hand, we do not require a strong bound of individual Hecke eigenvalue, but only the GRC \emph{on average}. 
In general, one expects \emph{Sarnak's Density Hypothesis} to hold; see \cite{sarnak1991diophantine,sarnak1991bounds,golubev2020sarnak,blomer2019density,jana2021application,assing2022density} for various aspects of this hypothesis. 

The hypothesis asserts that for every $\delta>0$, for a nice enough finite family $\calF$ of cusp forms, one has 
\begin{equation}\label{average-GRC}
\sum_{\varphi\in\calF}|\lambda_\varphi(p^l)|^2 \ll_\delta \left(p^{l}|\calF|\right)^{\delta}\left(|\calF| + p^{l(n-1)}\right)
\end{equation}
uniformly in $p,l$ and as $|\calF|\to \infty$.

Informally, the above says that \emph{larger Hecke eigenvalues occur with smaller density}.
Note that the above follows from GRC. On the other hand, the occurrence of the summand $p^{l(n-1)}$ represents as if the trivial eigenfunction appeared in $\calF$. The hypothesis above is an interpolation between the two cases. As a matter of fact, we expect that it will hold for natural families of \emph{discrete forms}, not only for \emph{cusp forms}. 

In this paper we will work on a specific kind of family $\calF$, namely, $\calF_T$ as defined in \cref{defn-ft}
whose cardinality is $\asymp T^{d}$ via \cref{eq:full Weyl law}. 
We will need a hypothesis in the following form.
\begin{conj}[Density hypothesis, Hecke eigenvalue version]
\label{conj:density conj}
Let $p$ be a fixed prime. For every $l\ge 0$ and $T\ge 1$ one has
\[\sum_{\varphi\in\calF_T} \left|\lambda_\varphi(p^l)\right|^2
\ll_{p,\delta}\left(Tp^l\right)^{\delta}\left(T^{d}+p^{l(n-1)}\right),\]
for every $\delta>0$
\end{conj}

\begin{remark}
\cref{conj:density conj} should be compared to the orthogonality conjecture which asserts that
\[\sum_{\varphi\in\calF_T}\lambda_\varphi(p_1^{l_1})\overline{\lambda_\varphi(p_2^{l_2})}\sim\delta_{p_1^{l_1}=p_2^{l_2}}T^d,\]
as $T\to\infty$; see \emph{e.g.}, \cite[Theorem 1, Theorem 7]{jana2021application}, \cite[Theorem 2]{blomer2019density} for more details.
\end{remark}

Motivated by the original density hypothesis of Sarnak (see \cite{sarnak1991diophantine}) one can propose an analogous density hypothesis in terms of the Langlands parameters in higher rank; see \cite{blomer2019density,jana2021application}. In fact, \cref{conj:density conj} is nothing but a reformulation of \emph{Sarnak's density hypothesis for higher rank} which we describe below.

\begin{conj}[Density Conjecture, Langlands parameter version]\label{conj:density conj alternative}
For every $0\le \theta_0 \le (n-1)/2$ and $T\ge 1$ one has
\[
|\{\varphi \in \calF_T\mid \theta_{\varphi,p} \ge \theta_0 \}| \ll_{\delta,p} T^{d \left(1-\frac{2\theta_0}{n-1}\right)+\delta},
\]
for every $\delta>0$
\end{conj}

Following \cite{blomer2019density} 
here we prove the equivalence of \cref{conj:density conj} and \cref{conj:density conj alternative}.

\begin{prop}\label{lem:density conjecture equivalence}
\cref{conj:density conj} and \cref{conj:density conj alternative} are equivalent.
\end{prop}

\begin{proof}
Assume that \cref{conj:density conj} holds. Given $T$ sufficiently large, summing over $l$ such that $p^{l(n-1)}\le T^d$ we get 
\[\sum_{\varphi\in\calF_T} \sum_{l: p^{l(n-1)} \le T^d}\left|\lambda_\varphi(p^l)\right|^2
\ll_{p,\delta}T^{d+\delta}.\] 
By \cite[Lemma~4]{blomer2019density}, for $k\ge n+1$ 
we have
\[
\sum_{l=0}^{k}\left|\lambda_\varphi(p^l)\right|^2 \gg_p p^{2k \theta_{\varphi,p}}.
\]
Therefore, for $T\ge p^{(n-1)(n+1)/d}$,
\[
\sum_{l:p^{l(n-1)}\le T^d}\left|\lambda_\varphi(p^l)\right|^2 \gg_p T^{d\frac{2\theta_{\varphi,p}}{n-1}},
\]
and
\[
\sum_{\varphi\in\calF_T} T^{d\frac{2\theta_{\varphi,p}}{n-1}}
\ll_{p,\delta}T^{d+\delta},\] 
and this implies \cref{conj:density conj alternative}.

For the other direction, assume \cref{conj:density conj alternative}. By \cref{lem:bounds on lambda from theta}, 
\[\sum_{\varphi\in\calF_T} \left|\lambda_\varphi(p^l)\right|^2 \ll_\delta \sum_{\varphi\in\calF_T} p^{l (2\theta_{\varphi,p}+ \delta)}.
\]
Using partial summation, this is bounded by the main term
\begin{align*}
&\ll_{p,\delta}\intop_{0}^{(n-1)/2} |\{\varphi \in \calF_T \mid \theta_{\varphi,p} \ge \theta_0 \}| p^{l(2\theta_0+\delta)}\d \theta_0 \\ &\ll_{p,\delta} (p^lT)^\delta\intop_{0}^{(n-1)/2} T^{d\left(1-\frac{2\theta_0}{n-1}\right)} p^{2l\theta_0}\d \theta_0
\end{align*}
plus the secondary terms
\begin{align*}
    |\calF_T|p^{l\delta}+ |\{\varphi \in \calF_T \mid  \theta_{\varphi,p}\ge (n-1)/2 \}|p^{l(n-1+\delta)}.
\end{align*}
For $0\le \theta_0 \le \frac{n-1}{2}$ we have
\[
T^{d\left(1-\frac{2\theta_0}{n-1}\right)}p^{2l\theta_0} \ll T^d+p^{l(n-1)}.
\]
Therefore, the main term is bounded by $(p^lT)^{\delta}(T^d+p^{l(n-1)})$. Similarly, using \cref{conj:density conj alternative}, the secondary terms are also bounded by the same value.
We therefore get \cref{conj:density conj}.
\end{proof}

\cref{conj:density conj} is actually a \emph{convexity estimate}, just as \cref{conj:density conj alternative}. As a matter of fact, one can replace in it $\calB_{n,\cusp}$ by $\calB_{n}$, as we will essentially show in \cref{subsec:pf- discrete spectrum} 
that the residual spectrum also satisfies this bound. 
However, for the cuspidal spectrum itself one should expect better than what \cref{conj:density conj} asserts, \emph{i.e.} the \emph{subconvex} estimates, which are indeed available for $n=2$ \cite[Lemma 1]{blomer2014sato} and $n=3$ \cite[Theorem~3.3]{buttcane2020plancherel}. For $n\ge 3$ Blomer \cite{blomer2019density} has proved a subconvex estimate for Hecke congruence subgroups in the level aspect. Recently, Blomer and Man \cite{blomer2022kloosterman} (improving upon the work of Assing and Blomer \cite{assing2022density}) proved subconvex estimates for principal congruence subgroups, again in the level aspect.

We describe the results for $n=2,3$ here which we will use latter. Although, we only need the convexity estimates for our proofs, we record the strongest known estimates. Let $m$ be of the form $p^l$ for some fixed prime\footnote{The results below are also true for the extension of $\lambda_\varphi$ to a multiplicative function on $\N$.} $p$.

\begin{prop}\label{subconvex-density-2}
Let $n=2$. We have
\[\sum_{\varphi\in\calF_T} \left|\lambda_\varphi(m)\right|^2
\ll_{\delta}\left(Tm\right)^{\delta}\left(T^{d}+m^{1/2}\right),\]
for every $\delta>0$.
\end{prop}

The result follows from \cite[Lemma~1]{blomer2014sato} and the discussion following it for the bound on $L(1,\operatorname{sym}^2u_j)$. It is stronger than the density estimate given in \cref{conj:density conj}. In fact, it leads to the following subconvex bound using the same methods as in \cref{lem:density conjecture equivalence}:
for every $0\le \theta_0 \le 1/2$ and $T\ge 1$ and we have
\begin{equation}\label{subconvex-density-alt-2}
|\{\varphi \in \calF_T \mid \theta_{\varphi,p} \ge \theta_0 \}| \ll_{\delta,p} T^{2 (1-4\theta_0)+\delta},
\end{equation}
for every $\delta>0$.
\cref{subconvex-density-alt-2} is slightly stronger than \cite[Proposition~1]{blomer2014sato}.

\begin{prop}\label{subconvex-density-3}
Let $n=3$. We have 
\begin{align*}
\sum_{\varphi\in\calF_T} \left|\lambda_\varphi(m)\right|^2
\ll_\delta \left(Tm\right)^{\delta}\left(T^{5}+m^{5/4}\right),
\end{align*}
for every $\delta>0$.
\end{prop}

The result can be obtained from \cite[Theorem~3.3]{buttcane2020plancherel} and upper bound of the adjoint $L$-value as in \cite[Equation~(22)]{blomer2014sato}\footnote{\cref{conj:density conj} for $n=3$ can actually be deduced from \cite{blomer2013applications}.}. This result can be used to prove a considerably stronger statement than \cref{conj:density conj alternative}.


\begin{prop}\label{subconvex-density-alt-3}
Let $n=3$ and $p$ fixed. For every $0\le \theta_0 \le 1$ and $T\ge 1$ we have
\[
|\{\varphi \in \calF_T \mid \theta_{\varphi,p} \ge \theta_0 \}| \ll_{\delta,p} T^{5 (1-8\theta_0/5)+\delta},
\]
for every $\delta>0$.
\end{prop}

\begin{proof}
We use the same arguments as in \cref{lem:density conjecture equivalence}. For $T$ large enough we sum the estimate in \cref{subconvex-density-3} with $m=p^l$ and $l$ such that $p^l \le T^4$. Then using \cite[Lemma 4]{blomer2019density} we get
\[
\sum_{\varphi\in\calF_T} T^{8\theta_{\varphi,p}}
\ll_{p,\varepsilon}T^{5+\varepsilon},
\]
and this implies the proposition.
\end{proof}
This substantially improves both \cite[Theorem~1]{blomer2014sato} and \cite[Theorem~2]{blomer2014sato}.

\subsection{Eisenstein Series}
\label{subsec:Eisenstein series}

Let $P$ be a standard parabolic in $G:=\PGL(n)$ 
attached to an ordered partition $n=n_1+\dots+n_r$. Let $M$ be the corresponding Levi subgroup and $N$ be the corresponding unipotent radical; see \cref{subsec:spherical representations}, where we denote a general parabolic by $Q$.

Let $T_M$ be the connected component of the identity of the $\R$-points in a maximal torus in the center of $M$.
Let $M(\bbA)^1$ 
be the kernel of all algebraic characters of $M$ (see \cite[Chapter~3]{arthur2005intro}). It holds that 
$M(\bbA)= M^1(\bbA)\times T_M$.
Following \cite[Chapter~7]{arthur2005intro} we give a brief sketch of the construction of the Eisenstein series on $G(\bbA)$, constructed inductively from the elements of $L^2_\disc(M(\Q) \backslash M(\bbA)^1)$. 

We follow the construction of induced representation as in \cref{subsec:spherical representations}. Given an representation $\pi$ of $M(\bbA)$ occurring in $L^2_\disc(M(\Q)\backslash M(\bbA)^1)$ and $\lambda\in\a_{P,\C}^*$ we consider the representation $\pi_\lambda:=\pi\otimes\chi_\lambda$ of $M(\bbA)$ and extend it to $P(\bbA)$ via the trivial representation on $N(\bbA)$. Then we consider the normalized parabolic induction \[\calI_{P,\pi}(\lambda):=\ind_{P(\bbA)}^{G(\bbA)}\pi_\lambda.\]
The representation will be unitary when $\lambda \in i\a_{P}^*$. 

When $\pi$ is the right regular representation then we denote $\calI_{P,\pi}$ as $\calI_{P}$. One can realize the $\calI_{P}(\lambda)$ on the Hilbert space $\calH_{P}$ (which is $\lambda$-independent) defined by the space of functions
\[\varphi\colon N(\bbA)M(\Q)T_M\backslash G(\bbA)\to \C\]
such that for every $x \in G(\bbA)$ the function
$\varphi_x\colon m \to \varphi(mx)$
belongs to $L^2_{\disc}(M(\Q)\backslash M(\bbA)^1)$, and
\[\|\varphi\|^2:=\int_{M(\Q)\backslash M(\bbA)^1}\int_{K_\bbA}|\varphi(mk)|^2\d k\d m <\infty.\]

Let $\calH_{P}^0\subset \calH_P$ be the subset of $K_\bbA$-finite vectors. For each element $\varphi\in\calH_{P}^0$ and $\lambda\in\a^*_{P,\C}$ we define an Eisenstein series as a function on $G(\Q)\backslash G(\bbA)$ by
\[\Eis_P(\varphi,\lambda)(x):=\sum_{\gamma\in P(\Q)\backslash G(\Q)}\varphi(\gamma x)\chi_{\rho_P+\lambda}(a(\gamma x)),\quad x\in G(\bbA).\]
The above sum converges absolutely for sufficiently dominant $\lambda$ and can be meromorphically continued for all $\lambda\in\a^*_{P,\C}$ by the work of Langlands. Moreover, $\calI_{P,\pi}(\lambda)$ intertwines with $\Eis_P$, in the sense that,
\[\Eis_P(\calI_{P}(\lambda)(g)\varphi,\lambda)(x)=\Eis_P(\varphi,\lambda)(xg),\quad x\in\X_\bbA, g\in G(\bbA).\]

In particular, if $\Eis_P(\varphi,\lambda)$ is $K_\bbA$-invariant then $\varphi\in\calH_{P}^{K_\bbA}$. Consequently, $\varphi$ can be realized as an element of $L^2_{\disc}(M(\Q)\backslash M(\bbA)^1)^{K_\bbA\cap M(\bbA)}$.

We have the decomposition 
\[\calH_{P}^{K_\bbA}=\bigoplus_{\pi}\calH_{P,\pi}^{K_\bbA},\]
where $\pi$ runs over the isomorphism classes of irreducible sub-representations of $M(\bbA)$ occurring in the subspace $L^2_\disc(M(\Q)\backslash M(\bbA)^1)$. 
By the Iwasawa decomposition $G(\bbA) = M(\bbA)N(\bbA)K_{\bbA}$, we have an isomorphism of vector spaces $\calH_P^{K_\bbA} \cong L^2_{\disc}(M(\Q)\backslash M(\bbA)^1)^{K_{M,\bbA}}\ne \{0\}$, where $K_{M,\bbA} = K_{\bbA}\cap M(\bbA)$. Similarly, we have $\calH_{P,\pi}
^{K_\bbA} \cong \pi^{K_M,\bbA}$, which is one-dimensional by the multiplicity one theorem. In this case we choose $\varphi_\pi \in \calH_{P,\pi}^{K_\bbA}$ to be a vector of norm $1$. We let $\calB_{P}$ be the set of such vectors $\varphi$, when we go over all the possible representation $\pi$ of $M(\bbA)$ appearing in $L^2_\disc(M(\Q) \backslash M(\bbA)^1)$ with $\pi^{K_{M,\bbA}}\ne 0$.

Let us describe the last set more explicitly. 
Each irreducible representation of $M(\bbA)$ appearing in the decomposition of $L^2_\disc(M(\Q)\backslash G(\bbA))$ has a central character $\chi$ of the center $Z(M(\bbA))$, trivial on $Z(M(\bbA)^1)\cap M(\Q)$. If the representation has a $K_{M,\bbA}$-invariant vector then $\chi$ must also be trivial on $Z(M(\bbA)^1)\cap K_{M,\bbA}$. 
Every central character of $\GL_{n_i}(\bbA)$ which is trivial on $\GL_{n_i}(\Q)$ and the maximal open compact subgroup $O_{n_i}(\R)\times \prod_p \GL_{n_i}(\Z_p)$ is trivial. 
Since $M$ is essentially a product of $\GL_{n_i}$, 
we deduce that $\chi$ is trivial. Therefore, $\calB_P$ is in bijection with irreducible subrepresentations of
\[
L^2_{\disc}((\PGL_{n_1}(\Q)\times\dots \times \PGL_{n_r}(\Q)) \backslash (\PGL_{n_1}(\bbA)\times\dots \times \PGL_{n_r}(\bbA))),
\]
having a $K_{n_1,\bbA}\times\dots \times K_{n_r,\bbA}$-invariant vector, where $K_{n_i,\bbA}$ is the maximal compact subgroup in $\PGL_{n_i}(\bbA)$.

The last space decomposes into a linear span of $\varphi_1\otimes \dots \otimes \varphi_r$, $\varphi_i \in \calB_{n_i}$.
We conclude that $\calB_P$ is in bijection with $\calB_{n_1}\times\dots \times \calB_{n_r}$.

If $\mu_{\varphi_i,v}$ is the Langlands parameter of $\varphi_i$ as the place $v$ then we embed it in $\a_{\C}^*$ in $(n_1+\dots +n_{i-1}+1,\dots ,n_1+\dots +n_i)$-th coordinates. Consequently, the Langlands parameter of $\Eis_P(\varphi,\lambda)$ at the place $v$ is 
\[\mu_{\varphi,\lambda, v} := (\mu_{\varphi_i,v},\dots, \mu_{\varphi_r,v})+\lambda \in \a_\C^*,\] which follows from \cref{langlands-paramteres-induction}.
In particular, for $\lambda \in i\a_{P}^*$ we have\footnote{Notice that $\theta_{\varphi,\lambda,p}$ does not depend on $\lambda$, so we discard it from the notations.}
\begin{equation}\label{eq:theta of Eis}
\theta_{\varphi,v}=\theta_{\varphi,\lambda,v}:=\theta(\mu_{\varphi,\lambda,p}) = \max_i{\theta_{\varphi_i,p}},   
\end{equation}
where $\theta_{\varphi,v}$ is as defined in \cref{defn-theta}.
We also have
\begin{equation}\label{eq:nu of Eis}
\nu_{\varphi,\lambda}^2:=\|\mu_{\varphi,\lambda,\infty}\|^2 = \nu_{\varphi_1}^2+\dots +\nu_{\varphi_r}^2+\|\lambda\|^2,
\end{equation}
where $\nu_\varphi$ is as defined in \cref{defn-nu}.
We also abbreviate (and slightly abuse notations) $\nu_{\varphi} = \nu_{\varphi,0}$.

If $h_v\in C_c^\infty(K_v \backslash \PGL_n(\Q_v) / K_v)$, then
\[R(h_v)\Eis_P(\varphi,\lambda) = \Eis_P(\calI_{P,\pi}(\lambda)(h_v)\varphi,\lambda)=\tilde{h}(\mu_{\varphi,\lambda,v})\Eis_P(\varphi,\lambda),\]
where $R(h_v)$ is defined as in \cref{subsec:adelic formulation}. 

In particular, if $f\in C_c^\infty(K_\bbA\backslash \PGL_n(\bbA) / K_\bbA)$ is of the form
\begin{equation}\label{eq:form of f}
    f(g) = f((g_v)_v) = f_\infty(g_\infty)f_p(g_p)\prod_{q\ne p} \One_{K_q}(g_q),
\end{equation}
then we have
\begin{equation}\label{eq:action on Eis}
    R(f)\Eis_P(\varphi,\lambda) =\Eis_P(\calI_{P,\pi}(h)\varphi,\lambda) \\
    = \tilde{f}_\infty(\mu_{\varphi,\lambda,\infty})
    \tilde{f}_p(\mu_{\varphi,\lambda,p})\Eis_P(\varphi,\lambda).
\end{equation}

\subsection{Spectral Decomposition}

We can describe Langlands spectral decomposition following \cite[Section~7]{arthur2005intro}.

We denote by $C_c^\infty(K_{\bbA}\backslash \PGL_n(\bbA) / K_{\bbA})$ the space that is spanned by functions of the form $f = \prod_{v\le \infty} f_v$, with $f_v \in C_c^\infty(K_v \backslash \PGL_n(\Q_v)/ K_v)$, and $f_v = \One_{K_v}$ for almost every $v$. 
Given $f\in C_c^\infty(K_{\bbA}\backslash \PGL_n(\bbA) / K_{\bbA})$, we will consider the operator 
\[R(f)\colon L^2(\X_\bbA)\to L^2(\X_\bbA),\]
defined by 
\begin{equation*}
    R(f)\varphi(x) 
    := \intop_{\PGL_n(\bbA)} f(y)\varphi(xy)  \d y 
    =  \intop_{\X_\bbA }K_f(x,y)\varphi(y)\d y,
\end{equation*} 
where 
\begin{equation}\label{defn-aut-kernel}
K_f(x,y) := \sum_{\gamma \in \PGL_n(\Q)} f(x^{-1}\gamma y).
\end{equation}
Note that the compact support of $f$ ensures that the above sum is finite.


Finally, we record the spectral decomposition of the automorphic kernel $K_f$.
\begin{equation}\label{eq: Spectral decomposition}
K_f(x,y) = \sum_P C_P \sum_{\varphi \in \calB_P} \intop_{i\a_P^*}  \Eis_P(\calI_{P,\pi_\varphi}(\lambda)(f)\varphi,\lambda)(x)\overline{\Eis_P(\varphi,\lambda)(y)} \d \lambda,
\end{equation}

Notice that when $P=G$ then $\calB_G=\calB_n$, there is no integral over $\lambda$, and $\Eis_P(\varphi,\lambda)$ is simply $\varphi$. The constants $C_P$ are certain explicit constants with $C_G=1$, and are slightly different than in \cite{arthur2005intro}, since we normalize the measure of $\a_\C$ differently.

We will also need the $L^2$-spectral expansion in the following form.
\begin{prop}\label{prop: spectral decomposition 2}
Let $f\in C_c^\infty(K_\bbA\backslash \PGL_n(\bbA) / K_\bbA)$ be as in \cref{eq:form of f}. 
For $x_0\in \X_{\bbA}$, let $F_{x_0} \in C_c^\infty(X_\bbA)$ be
\[
F_{x_0}(x) := K_f(x_0,x)
\]
Then 
\[
\|F_{x_0}\|_2^2 = \sum_P C_P \sum_{\varphi \in \calB_P}\intop_{i\a_P^*} |\tilde{f}_\infty(\mu_{\varphi,\lambda,\infty})|^2|\tilde{f}_p(\mu_{\varphi,\lambda,p})|^2 |\Eis_P(\varphi,\lambda)(x_0)|^2 \d \lambda,
\]
where the notations are as in \cref{eq: Spectral decomposition}.
\end{prop}

\begin{proof} 
Denote $f^*(g) := \overline{f(g^{-1})}$, and $f_1 := f*f^*$. 
We notice that for notations as in \cref{eq: Spectral decomposition}, using \cref{eq:action on Eis}
\begin{equation}\label{eq:sp pf 1}
\Eis_P(\calI_{P,\pi_\varphi}(f_1)\varphi,\lambda) = \Eis_P(\calI_{P,\pi_\varphi}(f)\calI_{P,\pi_\varphi}(f^*)\varphi,\lambda) = |\tilde{f}_\infty(\mu_{\varphi,\lambda,\infty})|^2|\tilde{f}_p(\mu_{\varphi,\lambda,p})|^2\Eis_P(\varphi,\lambda).    
\end{equation}
We next claim that 
\begin{equation}\label{eq:sp pf 2}
\|F_{x_0}\|_2^2 = K_{f_1}(x_0,x_0).
\end{equation}
Then the proof follows from \cref{eq: Spectral decomposition}, \cref{eq:sp pf 1} and \cref{eq:sp pf 2}.

To see \cref{eq:sp pf 2} we note that,
\begin{align*}
\|F_{x_0}\|_2^2 = \intop_{\X_\bbA} F_{x_0}(x)\overline{F_{x_0}(x)}\d x= \intop_{\X_\bbA}  \sum_{\gamma \in \PGL_n(\Q)} f(x_0^{-1}\gamma x) \sum_{\gamma_1 \in \PGL_n(\Q)}\overline{f(x_0^{-1}\gamma_1 x)}\d x
\end{align*}
Both of the above sums are finite. Exchanging order summation and integration, and unfolding the $\X_\bbA$ integral we obtain the above equals
\[ \intop_{\PGL_n(\bbA)} \sum_{\gamma \in \PGL_n(\Q)} f(x_0^{-1} \gamma g) f^{*}(g^{-1} x_0)\d g.\]
Once again exchanging finite sum with a compact integral we obtain
\[\sum_{\gamma \in \PGL_n(\Q)} (f*f^*)(x_0^{-1} \gamma x_0),\]
which concludes the proof.
\end{proof}

We will now use \cref{prop: spectral decomposition 2} to prove a version of a \emph{local weak Weyl Law}, which will be needed in our proof.

\begin{prop}\label{prop:Weyl Law local}
Let $\Omega\subset \X_A$ be a compact subset. Then for every $x_0 \in \Omega$ it holds that
\[
\sum_P C_P \sum_{\varphi\in \calB_P, \nu_\varphi \le T} \intop_{\lambda \in i\a_P^*, |\lambda| \le T} |\Eis_P(\varphi,\lambda)(x_0)|^2 \d \lambda \ll_\Omega T^{d},\]
as $T$ tends to infinity.
\end{prop}

\begin{proof}
Let $\varepsilon \asymp T^{-1}$ with a sufficiently small implied constant. Choose $f_\infty =k_\varepsilon$ where $k_\varepsilon$ is of the from described in \cref{lem:k_0 conditions}, and $f_p = \One_{K_p}$.
Construct $f_1$ and $F_{x_0}$ as in (the proof of) \cref{prop: spectral decomposition 2}. Note that $f_{1,p}$ is again $\One_{K_p}$ and $f_{1,\infty}$ is supported on $K_\infty B_{2\varepsilon} K_\infty$, which follows from property $(1)$ of \cref{lem:k_0 conditions}.
Then we claim that for $x_0 \in \Omega$
\[
\| F_{x_0}\|^2_2\,\, \ll_\Omega \|f_\infty\|_2^2\,\, \ll T^{d}.
\]
Notice that the second estimate follows from property $3$ in \cref{lem:k_0 conditions}.

First, we choose some fixed liftings of $x_0,x\in \PGL_n(\bbA)$, whose $p$-coordinates $x_{0,p},x_p$ are in $K_p$, and their $\infty$-coordinates are in a fixed fundamental domain of $\PGL_n(\Z) \backslash \PGL_n(\R)$. This is possible by \cref{lem: local-global space equivalence}.

If $f_1(x_0^{-1} \gamma x) \ne0$ it implies that $\One_{K_p}(x_{0,p}^{-1} \gamma x_p)\ne 0$ for all $p$. Hence, $\gamma \in K_p$ for all $p$, which implies that $\gamma\in\PGL_n(\Z)$. In addition,  $f_{1,\infty}(x_{0,\infty}^{-1}\gamma x_\infty)\ne 0 $, which implies that $x_{0,\infty}^{-1}\gamma x_\infty \in K_\infty B_{2\varepsilon} K_\infty$. 
Clearly, the number of $\gamma \in \PGL_n(\Z)$ with $x_{0,\infty}^{-1}\gamma x_\infty \in B_{2\varepsilon}$ is $\ll_\Omega 1$.

Using the proof of \cref{prop: spectral decomposition 2} we obtain
\[\| F_{x_0}\|^2_2\,\, \ll_\Omega \|f_{1,\infty}\|_{\infty}.\]
Applying Cauchy--Schwarz we see that the above is bounded by $\|f_\infty\|_2^2$, as needed.

First notice that for $\nu_\varphi \le T$ and $\|\lambda \|\le T$ it holds that 
\[\|\mu_{\varphi,\lambda,\infty}\| =\nu_{\varphi,\lambda} \ll T \asymp \varepsilon^{-1},\]
with a sufficiently small implied constant.
Thus using property $(5)$ of \cref{lem:k_0 conditions} we get that
\[|\tilde{f}_\infty(\mu_{\varphi,\lambda,\infty})| \gg 1.\]
Therefore we have
\begin{align*}
    &\sum_P C_P \sum_{\varphi\in \calB_P, \nu_\varphi \le T} \intop_{\lambda \in i\a_P^*, |\lambda| \le T} |\Eis_P(\varphi,\lambda)(x_0)|^2 \d \lambda \\
    \ll& \sum_P C_P \sum_{\varphi\in \calB_P} \intop_{\lambda \in i\a_P^*} |\tilde{f}_\infty(\mu_{\varphi,\lambda,\infty})|^2|\Eis_P(\varphi,\lambda)(x_0)|^2 \d \lambda.
\end{align*}
Applying \cref{prop: spectral decomposition 2} we conclude.
\end{proof}

\subsection{The Residual Spectrum and Shapes}
\label{subsec:residual sepctrum}
M{\oe}glin and Waldspurger in \cite{moeglin1989residual} described the residual spectrum $\calB_{n,\res}$ as follows. Let $ab=n$ with $b>1$ and let $\varphi \in \calB_{a,\cusp}$. 
Consider the standard parabolic subgroup $P$ corresponding to the ordered partition $(a,\dots ,a)$ of $n$, and let $\varphi'\in \calB_P$ correspond to $(\varphi,\dots ,\varphi)\in \calB_a\times\dots \times \calB_a$. Construct the Eisenstein series $\Eis_P(\varphi',\lambda)$ for $\lambda\in \a_{P,\C}^*$. This is a meromorphic function, and it has a (multiple) residue at the point 
\[\lambda = \rho_b = ((b-1)/2,(b-3)/2,\dots ,-(b-1)/2)\in \a_{P,\C}^*,\]
where we temporarily identified $\a_{P,\C}^*$ with a subset of $\C^b$ and recall that it embeds in $\a_\C^*$ by repeating each value $a$ times. The residue can be calculated as 
\[
\psi' = \lim_{\lambda\to \rho_b} (\lambda_1-\lambda_2-1)\cdot\dots \cdot(\lambda_{b-1}-\lambda_{b}+1)\Eis_P(\varphi',\lambda).
\]
After normalization, $\psi=\psi'/\|\psi'\|$ is an element of $\calB_{n,\res}$, and every element of $\calB_{n,\res}$ can be constructed this way for some $ab=n$, $b>1$, $\varphi\in \calB_a$. Thus we deduce that 
\begin{equation*}
\calB_{n}=\bigsqcup_{a\mid n}\calB_{a,\cusp}.
\end{equation*}

From the above description it follows that on the level of Langlands parameters we have
\[\mu_{\psi,v}= \mu_{\varphi',\rho_b,v} = 
(\mu_{\varphi,v},\dots ,\mu_{\varphi,v})+\rho_b.
\]
In particular, we have 
\[
\theta_{\psi,p}= \theta_{\varphi,p}+(b-1)/2.
\]
and if $\nu_\varphi\gg 1$ then $\nu_{\psi}\asymp \nu_\varphi$.

Each $\varphi \in \calB_P$ is parameterized by a \emph{shape} $(a_1,b_1),\dots,(a_r,b_r)$ where $a_i,b_i\ge1$ with $\sum_{i=1}^r a_ib_i=n$, and $P$ corresponds to the ordered  
partition $n=\sum_{i=1}^r a_ib_i$. Moreover, if $\varphi$ corresponds to $(\varphi_1,\dots ,\varphi_r)$, $\varphi_i\in \calB_{n_i}$, then $\varphi_i$ corresponds to a cuspidal representation in $\calB_{a_i,\cusp}$.

Given a shape $S=((a_1,b_1),\dots ,(a_r,b_r))$, we let $\calB_S \subset \calB_P$ be the set of forms $\varphi \in \calB_P$ of shape $S$. We have the following estimates.

\begin{lemma}\label{lem:residual theta bound}
For every $\varphi \in \calB_S$, it holds that
\[
\theta_{\varphi,p}\le \max_i\{(b_i-1)/2+\theta_{a_i,p}\},
\]
where $\theta_{a_i,p}$ is the best known bound towards the GRC at the place $p$ for the cuspidal spectrum of $\PGL(a_i)$.
\end{lemma}
Indeed, the lemma follows from the claim above about the behavior of $\theta$ of residual forms and \cref{eq:theta of Eis}.

The following estimate combines Weyl's law for a given shape.
\begin{lemma}\label{lem:Weyl law residual}
We have
\[
|\{\varphi \in \calB_S \mid \nu_\varphi \le T\}| \ll T^{d_S},
\]
where $d_S = \sum_{i=1}^r (a_i+2)(a_i-1)/2$.
\end{lemma}

\begin{proof}
As described above, there is a bijection between $\calB_S$ and $\B_{a_1,\cusp}\times\dots \times \calB_{a_r,\cusp}$. Moreover, under this bijection, by combining the estimates for $\nu$ of residual forms above and \cref{eq:nu of Eis} we have 
\[
\nu_\varphi +1 \asymp \max_{i=1}^r{\nu_{\varphi_i}}+1.
\]
The estimate now follows from \cref{prop:Weyl law upper bound}.
\end{proof}

\subsection{Local \texorpdfstring{$L^2$}{L2}-Bounds of Eisenstein Series}
\label{subsec:L2 bounds on Eisenstein}

Let $\varphi\in \calB_P$ and $\lambda \in i\a_P^*$, and let $\Eis_P(\varphi,\lambda) \in C^\infty(\X)$ be the corresponding Eisenstein series. It is known that the Eisenstein series grow polynomially near the cusp. It is a natural and challenging problem to find \emph{good} pointwise upper bound of $\Eis_P(\varphi,\lambda)$. 

A more tractable approach is to take a compact subset $\Omega\subset \X$ of positive measure and ask the size of $\|\Eis_P(\varphi,\lambda)|_\Omega\|_{2}$
as $\nu_{\varphi,\lambda}\to\infty$. In this paper we need an upper bound of $\|\Eis_P(\varphi,\lambda)|_\Omega\|^2_{2}$ on an average over $\lambda$ in a long interval. 
One may deduce certain bounds of such an average from the local Weyl law (cf. \cref{prop:Weyl Law local}) or via the improved $L^\infty$-bounds in \cite{blomer2016supnorm-pgln}, but such bounds are not sufficient for our purposes. 

One expects that $\|\Eis_P(\varphi,\lambda)|_\Omega\|_{2}$ remains essentially bounded in $\nu_{\varphi,\lambda}$, which is an analogue of the Lindel\"of hypothesis for the $L$-functions.
More precisely, we expect that 
\begin{equation}\label{conj:pointwise}
   \intop_{\Omega}|\Eis_P(\varphi,\lambda)(x)|^2 \d x \ll_{\Omega} \log^{n-1}(1+\nu_{\varphi,\lambda}).
\end{equation}
Note that for $n=2$ the above is classically known. We refer to \cite{jana2022eisenstein-average} for a detailed discussion.

We remark that a recent result of Assing--Blomer \cite[Theorem 1.5]{assing2022density} on optimal lifting for $\SL_n(\Z/q\Z)$ also requires such a bound on the local $L^2$-growth but in a non-archimedean aspect; see \cite[Hypothesis 1]{assing2022density}.

Proving \cref{conj:pointwise} seems to be quite difficult. A natural way to approach the problem is via the higher rank \emph{Maass--Selberg relations} due to Langlands, at least when $\varphi$ is \emph{cuspidal}. Among many complications that one faces through this approach (see \cite[\S 1.3]{jana2022eisenstein-average}) the major one involves \emph{standard (in $\nu_{\varphi,\lambda}$ aspect)} zero-free region for various $\GL(n)\times\GL(m)$ Rankin--Selberg $L$-functions, which are available only in a very few cases. In a forthcoming work \cite{jana2022eisenstein-pointwise} we give a conditional result towards what expected in \cref{conj:pointwise}. 

A relatively easy problem would be to find upper bound of a short $\lambda$ average of $\|\Eis_P(\varphi,\lambda)|_\Omega\|^2_2$. In fact, we expect
\begin{equation}\label{conj:average}
   \intop_{\|\lambda'-\lambda\|\le1}\intop_{\Omega}|\Eis_P(\varphi,\lambda')(x)|^2 \d x \d\lambda' \ll_{\Omega} \log^{n-1}(1+\nu_{\varphi,\lambda}).
\end{equation}
In a companion paper \cite{jana2022eisenstein-average} we study these problems in detail for a general reductive groups. 

For the current paper we only need to find an upper bound of $\|\Eis_P(\varphi,\lambda)|_\Omega\|^2_2$ on an average over $\lambda$ over a long interval so that the bound in the $\varphi$ aspect is only polynomial in $\nu_\varphi$ with very small degree. However, it is important for us that the bound is uniform over all $\varphi$, cuspidal or not. We describe the required estimate below.

\begin{prop}\label{thm:maass-selberg}
For every $\varphi \in \calB_P$ we have
\[
\intop_{\lambda \in i\a_P^*, \|\lambda\|\le T}\intop_{\Omega}|\Eis_P(\varphi,\lambda)(x)|^2 \d x \d \lambda \ll_{\Omega} \log(1+T+\nu_{\varphi})^{n-1}T^{\dim \a_P},
\]
for any $T\ge 1$.
\end{prop}
The theorem should be compared to the simple upper bounds using the local Weyl law as in \cref{prop:Weyl Law local}), which allows us to deduce a similar statement, but with $T^{\dim \a_P}$ replaced by $T^d$, which is insufficient for our purpose.

We remark that for $n=3$, \cref{thm:maass-selberg} was proved by Miller in \cite{miller2001existence}, as one of the main ingredients in proving Weyl's law for $\SL_3(\Z) \backslash \SL_3(\R)/\SO_3(\R)$. Therefore, for $n=3$ (and $n=2$) the results of this paper are unconditional on the result of the companion paper. 
\cref{thm:maass-selberg} generalizes Miller's result to higher rank, and similarly implies Weyl's law in the same way. However, we heavily rely on \cite{muller2007weyl}, so this does not lead to a new proof.

\begin{proof}[Proof of \cref{thm:maass-selberg}]
We find $\{\eta_j\}_{j=1}^k\in i\a_P^*$ with $k\ll T^{\dim\a_P}$ and $\|\eta_j\|\le T$ so that
\[\{\lambda\in i\a_P^*\mid \|\lambda\| \le T\} \subset \cup_{j=1}^k \{\lambda\in i\a_P^*\mid \|\lambda-\eta_j\| \le 1\}.\]
Clearly, we can majorize the integral in the proposition by 
\[\sum_{j=1}^k \intop_{\substack{\lambda \in i\a_P^*\\ \|\lambda-\eta_j\|\le 1}}\intop_{\Omega}|\Eis_P(\varphi,\lambda)(x)|^2 \d x \d \lambda.\]
We apply \cite[Theorem 1]{jana2022eisenstein-pointwise} to each summand on the right hand side above with $\varphi_0=\varphi$ and $\lambda_0=\eta_j$ to conclude that the integral in the proposition is bounded by
\[\ll_{\Omega} k \max_j \log(1+\nu_\varphi+\|\eta_j\|)^{n-1}.\]
We conclude the proof by employing the bounds on $\eta_j$ and $k$.
\end{proof}

\section{Reduction to a Spectral Problem}
\label{sec:reduction to spectral}
In this section, we finally begin proving the main results \cref{thm:kappa theorem intro} and \cref{thm:optimal-SDH}, and we will reduce it to a spectral problem. Then we will apply a few different reductions to simplify the problem even further. 

Consider the set 
\[S(k):= \{\gamma \in \SL_n(\Z[1/p])\mid \Ht(\gamma)\le k\}.\]

\begin{lemma}
The image of $S(k)$ in $\PGL_n(\Q)$ is equal to $\tilde{R}(p^{nk})$.
\end{lemma}

\begin{proof}
The map $S(k)\to R(p^{nk})$ defined by $\gamma\mapsto p^k\gamma$ is a bijection. Therefore, the images of them in $\PGL_n(\Q)$ are the same.
\end{proof}

Below we identify $\X := \SL_n(\Z) \backslash \H^n \cong \PSL_n(\Z) \backslash \PGL_n(\R) / K_\infty$. This implies that the action of $\gamma \in \SL_n(\R)$ on $\H^n$ depends only on the image of $\gamma$ in $\PGL_n(\R)$.

\begin{defn}\label{def:admissible}
Let $x,x_0 \in \X$ which we identify with some lifts of them in $\H^n \cong \PGL_n(\R) / K_\infty$. Also, let $\varepsilon>0$ and $k\in\Z_{\ge 0}$. We say that the pair $(x,x_0)$ is \emph{$(\varepsilon,k)$-admissible} if there is a solution $\gamma \in \tilde{R}(p^{nk})$ to $\dist(x,\gamma x_0)\le \varepsilon$.
\end{defn}

Notice that since $\tilde{R}(p^{nk})$ is left and right $\PSL_n(\Z)$-invariant, the above definition does not depend on the lifts of $x,x_0 \in \H^n$.

Unraveling \cref{def:Diopnatine exponents intro} and \cref{def:admissible} we get the following.
\begin{lemma}\label{lem:alternative-kappa-def}
Let $x,x_0 \in \X$. The Diophantine exponent $\kappa(x,x_0)$ is the infimum over $\zeta< \infty$ such that there exists $\varepsilon_0 = \varepsilon_0(x,x_0,\zeta)$ with the property that for every $\varepsilon<\varepsilon_0$ the pair $(x,x_0)$ is $(\varepsilon,\zeta\frac{n+2}{2n}\log_p(\varepsilon^{-1}))$-admissible.
\end{lemma}

Let $k_\varepsilon \in C_c^\infty(K_\infty \backslash \PGL_n(\R) / K_\infty)$ be as in \cref{lem:k_0 conditions}. For $x_0 \in \X$, let $K_{\varepsilon,x_0}^{\X}\in C_c^\infty(\X)$ be the automorphic kernel
\[
K_{\varepsilon,x_0}^{\X}(y) := \sum_{\gamma \in \PSL_n(\Z)} k_\varepsilon (x_0^{-1} \gamma^{-1} y).
\]
It is simple to see that 
\[
\intop_{\X} K_{\varepsilon,x_0}^{\X}(y) \d y = \intop_{\PGL_n(\R)} k_\varepsilon(x_0^{-1} g) \d g = 1.
\]

Recall the Hecke operator $T^*(p^k)$ from \cref{subsec:adelic formulation}, and the fact that it acts on functions on $\X$ by \cref{rem:hecke on X}.

\begin{lemma}\label{lem:non-zero to solving equation}
Assume that 
%
\[T^*(p^{nk_1})T^*(p^{nk_2})K_{\varepsilon,x_0}^\X(x)\ne 0,\]
then the pair $(x,x_0)$ is $(\varepsilon,k_1+k_2)$-admissible.
\end{lemma}

\begin{proof}
%
We have
\[
0\neq T^*(p^{nk_1})T^*(p^{nk_2})K_{\varepsilon,x_0}^{\X}(x) = \sum_{\gamma_1\in A(p^{nk_1})}\sum_{\gamma_2\in A(p^{nk_2})}\sum_{\gamma\in \PSL_n(\Z)} k_{\varepsilon}(x_0^{-1}\gamma^{-1}\gamma_2^{-1}\gamma_1^{-1}x).
\]
So there is a $\gamma':= \gamma_1 \gamma_2 \gamma$ such that $k_{\varepsilon}(x_0^{-1}\gamma^{\prime-1}x)\ne 0$. It holds that $\gamma' \in \tilde{R}(p^{n(k_1+k_2)})$. 
By the assumption of the support of $k_\varepsilon$, we have 
\[
\dist(x_0^{-1}\gamma^{\prime-1}x,e)= \dist(x,\gamma' x_0)\le \varepsilon,
\]
as needed.
\end{proof}

Let $\pi_\X = \frac{\One_\X}{m(\X)}$ be the $L^1$-normalized characteristic function on $\X$. 

\begin{lemma}\label{lem:L2 to solving equation}
Let $x_0\in\X$. Assume that 
\[
\|T^*(p^{nk_1})T^*(p^{nk_2})K_{\varepsilon,x_0}^{\X}-\pi_\X\|_2 \,\,\le \frac{c}{\sqrt{m(\X)}}.
\]
Then there is a subset $Y \subset \X$, such that 
\[
m(Y)\ge m(\X)(1-c^2),
\]
such that for all $x \in Y$ the pair $(x,x_0)$ is $(\varepsilon,k_1+k_2)$-admissible.
\end{lemma}

\begin{proof}
Let $Y := \{x\in \X \mid T^*(p^{nk_1})T^*(p^{nk_2})K_{\varepsilon,x_0}^{\X}(x)\ne 0\}$. By \cref{lem:non-zero to solving equation}, each $x\in Y$ is $(\varepsilon,k_1+k_2)$-admissible. On the other hand, 
\[
\|T^*(p^{nk_1})T^*(p^{nk_2})K_{\varepsilon,x_0}^{\X}-\pi_\X\|_2^2\,\, \ge \intop_{\X\setminus Y} \pi_\X(x)^2 \d x = m(\X \setminus Y) / m(\X)^2.
\]
We deduce that $m(\X\setminus Y)\le c^2m(\X)$, as needed.
\end{proof}


\begin{lemma}\label{lem:kappa from L2 local}
Let $x_0\in\X$ and $\beta\ge 1$. Assume that there is $\alpha>0$ such that for every $\delta>0$ there is $\varepsilon_0>0$, such that for every $0<\varepsilon<\varepsilon_0$ there are $k_1,k_2$ with $k_1+k_2\le (1+\delta)\beta \frac{n+2}{2n}\log_p(\varepsilon^{-1})$, such that we have
\[
\|T^*(p^{nk_1})T^*(p^{nk_2})K_{\varepsilon,x_0}^{\X}-\pi_\X\|_2\,\, \le \varepsilon^{\alpha \delta}.
\]
Then $\kappa(x_0)\le \beta$.
\end{lemma}
\begin{proof}
Let $\delta>0$. 
For $\varepsilon$ fixed, let $Z_{\varepsilon,\delta}\subset \X $ be the set of $x\in \X$ such that the pairs $(x,x_0)$ are not $(\varepsilon,k)$-admissible with $k \le (1+\delta)\beta\frac{n+2}{2n}\log_p(\varepsilon^{-1})$. 
Using \cref{lem:alternative-kappa-def} it suffices to prove that for almost every $x\in \X$, for $\varepsilon_0$ small enough depending on $x,\delta$, and $\varepsilon<\varepsilon_0$ we have $x\notin Z_{\varepsilon,\delta}$.

Let $\varepsilon_j := e^{-cj}$, for some $c>0$ sufficiently small relatively to $\delta$. Then for $\varepsilon$ small enough, there is $\varepsilon_j$ such that $Z_{\varepsilon,\delta} \subset Z_{\varepsilon_j,\delta/2}$. Therefore, it suffices to prove that for almost every $x\in \X$, for $m\in \Z_{\ge 0}$ large enough, $x\notin Z_{\varepsilon_j,\delta/2}$. Using the Borel--Cantelli lemma it is enough to prove that 
\begin{equation}\label{eq: kappa from L2 proof}
\sum_j m(Z_{\varepsilon_j,\delta/2})<\infty.    
\end{equation}
By the assumption and \cref{lem:L2 to solving equation}, there is $\varepsilon_0>0$ such that for $\varepsilon_j<\varepsilon_0$,
\[ m(Z_{\varepsilon_j,\delta/2})\ll \varepsilon_j^{2\alpha \delta}=e^{-2c\alpha\delta j}.\]
This shows that \cref{eq: kappa from L2 proof} holds, as needed.
\end{proof}

We now add an additional average over $x_0$. Let $\Omega\subset \X$ be a fixed compact subset of positive measure. 

\begin{lemma}\label{lem:kappa from L2 global}
Let $\beta\ge 1$. Assume that there is $\alpha>0$ such that for every $\delta>0$ there is $\varepsilon_0>0$, such that for every $0<\varepsilon<\varepsilon_0$ there are $k_1,k_2$ with $k_1+k_2\le (1+\delta)\beta \frac{n+2}{2n}\log_p(\varepsilon^{-1})$, such that we have
\[
\intop_\Omega \|T^*(p^{nk_1})T^*(p^{nk_2})K_{\varepsilon,x_0}^{\X}-\pi_\X\|_2^2 \d x_0 \le \varepsilon^{\alpha \delta}.
\]
Then $\kappa\le \beta$.
\end{lemma}
\begin{proof}

Let $\delta>0$. Let $Z_{\varepsilon,\delta} \subset \X \times \X$ be the set of $(x,x_0)\in \X\times \Omega$ that are not $(\varepsilon,k)$-admissible, for $k \le  (1+\delta)\beta\frac{n+2}{2n}\log_p(\varepsilon^{-1})$. 

Since $\kappa = \kappa(x,x_0)$ for almost every $x,x_0$ (see \cref{sec:GGN}), using \cref{lem:alternative-kappa-def} it suffices to prove that for almost every $x_0 \in \Omega$ and almost every $x \in \X$, there is an $\varepsilon_0$ such that for $\varepsilon<\varepsilon_0$ the pair
$(x,x_0)\notin Z_{\varepsilon,\delta}$.
Using the same argument as in the proof of \cref{lem:kappa from L2 local}, we may assume that $\varepsilon = \varepsilon_j= e^{-cj}$, and using Borel--Cantelli it is enough to prove that 
\[
\sum_j m(Z_{\varepsilon_j,\delta})< \infty.
\]
For $\varepsilon<\varepsilon_0$ small enough, let $Y_{\varepsilon,\delta}\subset \Omega$ be the set of $x_0\in \Omega$ such that for $k$ as in the assumption of the lemma
\[
\|T^*(p^{nk_1})T^*(p^{nk_2})K^\X_{\varepsilon,x_0}-\pi_\X\|_2^2\,\, \le \varepsilon^{\alpha \delta/2}.
\]
We claim that $m(\Omega-Y_{\varepsilon,\delta})\le \varepsilon^{\alpha \delta/2}$. 
Indeed,
\[
\intop_{\Omega-Y_{\varepsilon,\delta}} \|T^*(p^{nk_1})T^*(p^{nk_2})K^\X_{\varepsilon,x_0}-\pi_\X\|_2^2 \d x_0 \ge \varepsilon^{\alpha\delta/2} m(\Omega - Y_{\varepsilon,\delta}),
\]
so
$
\varepsilon^{\alpha\delta/2} m(\Omega - Y_{\varepsilon,\delta}) \le \varepsilon^{\alpha\delta}
$
giving the desired estimate.

Now, for $x_0\in Y_{\varepsilon,\delta}$ by \cref{lem:L2 to solving equation} we have
\[
m(\{x\in \X \mid (x,x_0) \in Z_{\varepsilon,\delta}\}) \ll \varepsilon^{\alpha \delta/2}.
\]
Therefore, 
\[
m(Z_{\varepsilon,\delta}) \le 
m(\Omega - Y_{\varepsilon,\delta})m(\X) + m(Y_{\varepsilon,\delta}) \varepsilon^{\alpha \delta/2} \ll \varepsilon ^{\alpha \delta/2}.
\]
Using the last estimate we get 
\[
\sum_j m(Z_{\varepsilon_j,\delta})< \infty,
\]
as needed.
\end{proof}

We now discuss a further reduction, which allows us prove bounds of $\kappa$ but with weaker assumptions than that of \cref{lem:kappa from L2 local} and \cref{lem:kappa from L2 global}.
First, we will need the following estimates, which play a major role in the work \cite{ghosh2018best}, and which we already discussed in \cref{prop: GGN spectral gap} using different notations.

\begin{lemma}\label{lem:spectral gap}
For all $n\ge 2$ there is an $\alpha>0$ such that as an operator on $L^2_0(\X)$
\[\|T^*(p^l)\|_{\mathrm{op}}\ll p^{-l\alpha}.\]
Moreover, for $n\ge 3$ any $\alpha<1/2$ and for $n=2$ (resp. under the GRC) any $\alpha<25/64$ (resp. $\alpha<1/2$) work.
\end{lemma}

\begin{proof}
Using \cref{rem:hecke on X}, the proof follows from bounds on the integrability exponents of the action of $\PGL_n(\Q_p)$ on $L^2(\PGL_n(\Q) \backslash \PGL_n(\bbA))$, as in \cref{prop: GGN spectral gap} and the discussions after it. 
\end{proof}

Some remarks are in order now.
\begin{enumerate}
    \item By combining \cref{lem:kappa from L2 local} and \cref{lem:spectral gap} we may deduce \cref{thm:GGN}. Indeed, we essentially recovered the arguments in \cite{ghosh2018best} for our specific case.
    \item For $n\ge 3$, by \cite[Theorem~1.5]{clozel2001hecke}, \cref{lem:spectral gap} is optimal in the sense that for every $\delta>0$ there exists $f\in L^2_0(\X)$, with 
    \[
    \|T^*(p^l) f\|_2 \gg_{\delta} p^{-l(1/2+\delta)}\|f\|_2.
    \]
    This shows that to prove \cref{thm:kappa theorem intro} one needs stronger tools than spectral gap alone.
\end{enumerate}

\cref{lem:spectral gap} allows us to give the following versions of \cref{lem:kappa from L2 local} and \cref{lem:kappa from L2 global}.

\begin{lemma}\label{lem:kappa from L2 local convex X}
Let $\beta\ge 1$. Assume that there is an $\varepsilon_0>0$ such that for every $0<\varepsilon<\varepsilon_0$ there is $k\le \beta \frac{n+2}{2n}\log_p(\varepsilon^{-1})$ such that we have
\[
\|T^*(p^{nk})K_{\varepsilon,x_0}^{\X}\|_2 \ll_\eta \varepsilon^{-\eta},
\]
for every $\eta>0$.
Then $\kappa(x_0)\le \beta$.
\end{lemma}

\begin{proof}
By the assumption, there is an $\varepsilon_0$ such that for $\varepsilon<\varepsilon_0$ and for some $k_2 \le\beta\frac{n+2}{2n}\log_p(\varepsilon^{-1})$, it holds that 
\[
\|T^*(p^{nk_2})K_{\varepsilon,x_0}^{\X}\|_2 \ll_\eta \varepsilon^{-\eta}.
\]
Since $T^*(p^{nk_2})$ is an average operator and $\intop_\X K_{\varepsilon,x_0}^{\X}(x) \d x =1$ we have
\[T^*(p^{nk_2})K_{\varepsilon,x_0}^{\X}-\pi_\X \in L^2_0(\X).\]
Let $\delta>0$. Let $k_1 = \lfloor \beta \delta\frac{n+2}{2n}\log_p(\varepsilon^{-1})\rfloor$. Notice that $k_1+k_2 \le \beta (1+\delta)\frac{n+2}{2n}\log_p(\varepsilon^{-1})$. 

Applying Lemma~\Ref{lem:spectral gap} we find some $\alpha>0$ such that
\begin{align*}
\| T^*(p^{nk_1})T^*(p^{nk_2})K_{\varepsilon,x_0}^{\X} -\pi_\X\|_2 &= \| T^*(p^{nk_1})(T^*(p^{nk_2})K_{\varepsilon,x_0}^{\X} -\pi_\X)\|_2 \\
&\ll \varepsilon^{\alpha\delta}\|T^*(p^{nk_2})K_{\varepsilon,x_0}^{\X} -\pi_\X\|_2 \ll_\eta \varepsilon^{\alpha\delta-\eta}.
\end{align*}
By choosing $\varepsilon_0^\prime,\eta$ small enough, 
for $\varepsilon<\varepsilon_0^\prime$ the above is $\le \varepsilon^{\alpha\delta/2}$. The lemma now follows from \cref{lem:kappa from L2 local}.
\end{proof}

Our final reduction will allow us to replace the space $\X$ by the nicer space $\X_0$.
For $x_0\in \X_0$, let $K_{\varepsilon,x_0}^{\X_0}\in C_c^\infty(\X_0)$ be
\[
K_{\varepsilon,x_0}^{\X_0}(y) :=\sum_{\gamma \in \PGL_n(\Z)}k_\varepsilon(x_0^{-1}\gamma^{-1}y).
\]
Let $\Phi\colon \X \to \X_0$ be the covering map. Then we can define a push-forward map $\Phi_*\colon L^2(\X) \to L^2(\X_0)$, defined for $f\in L^2(\X)$, $y\in \X$ as 
\[
\Phi_*(f)(\Phi(y)) := \sum_{\gamma \in \PGL_n(\Z)/\PSL_n(\Z)} f(\gamma y).
\]
We have the simple norm estimate on push-forward maps for \emph{non-negative} $f$,
\begin{equation}\label{eq:push-forward}
    \|f\|_{2}\,\, \le \|\Phi_*f\|_{2},
\end{equation}
where on the left-hand side the norm on $L^2(\X)$ and on the right-hand side the norm is on $L^2(\X_0)$.
\begin{lemma}\label{lem:X and X_0 comparison}
Let $x_0 \in \X$. Then it holds that 
\[
K_{\varepsilon,\Phi(x_0)}^{\X_0} = \Phi_* K_{\varepsilon,x_0}^{\X}
\]
and similarly
\[
T^*(p^{nk})K_{\varepsilon,\Phi(x_0)}^{\X_0} = \Phi_*( T^*(p^{nk})K_{\varepsilon,x_0}^{\X})
\]
\end{lemma}
\begin{proof}
It is sufficient to prove the second estimate. Indeed, unwinding the definitions we get that 
\[
T^*(p^{nk})K_{\varepsilon,x_0}^{\X}(y) = \sum_{\gamma \in \tilde{R}(p^{nk})}k_\varepsilon(x_0^{-1}\gamma^{-1}y),
\]
and similarly,
\[
T^*(p^{nk})K_{\varepsilon,x_0}^{\X_0}(y) = \sum_{\gamma'\in\PGL_n(\Z)/\PSL_n(\Z)}\sum_{\gamma \in \tilde{R}(p^{nk})}k_\varepsilon(x_0^{-1}\gamma^{-1}\gamma'y).
\]
The lemma follows.
\end{proof}

Finally, combining \cref{lem:kappa from L2 local convex X}, \cref{lem:X and X_0 comparison}, and \cref{eq:push-forward}, we deduce the following.

\begin{lemma}\label{lem:kappa from L2 local convex}
Let $\beta\ge 1$ and $x_0\in\X$. Assume that there is an $\varepsilon_0>0$ such that for every $0<\varepsilon<\varepsilon_0$ and for some $k\le \beta \frac{n+2}{2n}\log_p(\varepsilon^{-1})$ we have
\[
\|T^*(p^{nk})K_{\varepsilon,\Phi(x_0)}^{\X_0}\|_2 \ll_\eta \varepsilon^{-\eta},
\]
for every $\eta>0$.
Then $\kappa(x_0)\le \beta$.
\end{lemma}

The same set of arguments, with \cref{lem:kappa from L2 global} in place of \cref{lem:kappa from L2 local} will give:
\begin{lemma}\label{lem:kappa from L2 global convex}
Let $\Omega \subset \X_0$ be a compact set of positive measure. Let $\beta\ge 1$. Assume that for every $\delta>0$ there is $\varepsilon_0>0$ such that for every $0<\varepsilon<\varepsilon_0$ there is $k\le (1+\delta)\beta \frac{n+2}{2n}\log_p(\varepsilon^{-1})$ such that
\[
\intop_\Omega \|T^*(p^{nk})K_{\varepsilon,x_0}^{\X_0}\|_2^2 \d x_0 \ll_\eta \varepsilon^{-\eta},
\]
for every $\eta>0$.
Then $\kappa\le \beta$.
\end{lemma}

\section{Applying the Spectral Decomposition}
\label{sec:applying spectral decomposition}

Consider the adelic function $f \in C_c^\infty(K_\bbA \backslash \PGL_n(\bbA) / K_\bbA)$ defined by
\[f((g)_v):= k_\varepsilon(g_\infty)
h_{p^{nk}}(g_p)\prod_{q\ne p} \One_{K_q}(g_q),
\]
where $h_{p^{nk}}$ is as in \cref{defn-hpl} and $k_\varepsilon$ is given by \cref{lem:k_0 conditions}.

Given $x_0 \in \X_0$, identify it by a slight abuse of notations as an element $x_0\in \X_\bbA$. 
Consider the function 
\[
F_{x_0}(x) = \sum_{\gamma \in \PGL_n(\Q)} f_1(x_0^{-1} \gamma^{-1} x),
\]
where $f_1$ is the self-convolution of $f$, as defined in the proof of \cref{prop: spectral decomposition 2}.
Using the discussion in \cref{subsec:adelic formulation}, we see that $F_{x_0}$ is the adelic version of the function $T^*(p^{nk})K_{\varepsilon,x_0}^{\X_0}$.  
Therefore,
\[
\|T^*(p^{nk})K_{\varepsilon,x_0}^{\X_0}\|_2^2 =
\|F_{x_0}\|_2^2,
\]
where the underlying space on the left-hand side is $\X_0$ and on the right-hand side is $\X_\bbA$.

We apply \cref{prop: spectral decomposition 2} to $F_{x_0}$, and obtain that 
\begin{align}\label{eq:basic spectral decomposition}
\|T^*(p^{nk})K_{\varepsilon,x_0}^{\X_0}\|_2^2 \,\,=& \sum_{P} C_P \sum_{\varphi\in \calB_P}\intop_{i\a_P^*} |\tilde{k}_\varepsilon(\mu_{\varphi,\lambda,\infty})|^2|\tilde{h}_{p^{nk}}(\mu_{\varphi,\lambda,p})|^2 |\Eis_P(\varphi,\lambda)(x_0)|^2 \d \lambda \\   
=& \sum_{\varphi\in \calB_G} |\tilde{k}_\varepsilon(\mu_{\varphi,\infty})|^2|\tilde{h}_{p^{nk}}(\mu_{\varphi,p})|^2 |\varphi(x_0)|^2  \nonumber \\
+&\sum_{P\ne G} C_P \sum_{\varphi\in \calB_P}\intop_{i\a_P^*} |\tilde{k}_\varepsilon(\mu_{\varphi,\lambda,\infty})|^2|\tilde{h}_{p^{nk}}(\mu_{\varphi,\lambda,p})|^2 |\Eis_P(\varphi,\lambda)(x_0)|^2 \d \lambda \nonumber
\end{align}
Using \cref{lem:k_0 conditions} we get that for every $N> 0$, 
\begin{equation}\label{useful-bound-tilde-k}
|\tilde{k}_\varepsilon(\mu_{\varphi,\lambda,\infty})| \ll_N (1+\varepsilon\nu_{\varphi,\lambda})^{-N} \ll  (1+\varepsilon\nu_{\varphi})^{-N/2}(1+\varepsilon\|\lambda\|)^{-N/2},\quad |\tilde{k}_\varepsilon(\mu_{\varphi,\infty})| \ll_N (1+\varepsilon\nu_{\varphi})^{-N}.
\end{equation} 
Also using \cref{lem:bounds on lambda from theta} we get that for every $\eta>0$,
\begin{equation}\label{useful-bound-h}
|\tilde{h}_{p^{nk}}(\mu_{\varphi,\lambda,p})| \ll_\eta p^{kn(\theta_{\varphi,p}-(n-1)/2+\eta)},\quad |\tilde{h}_{p^{nk}}(\mu_{\varphi,p})| \asymp p^{-kn(n-1)/2}|\lambda_\varphi(p^{nk})|.
\end{equation}
Thus we arrive at the following proposition.

\begin{prop}[Truncation]\label{prop:truncation local}
Let $\Omega\subset \X_0$ be a fixed compact set. Then for every $x_0\in \Omega$ and $\delta,\eta>0$ we have
\begin{multline*}
    \|T^*(p^{nk})K_{\varepsilon,x_0}^{\X_0}\|_2^2 \,\,\ll_{\Omega,N,\eta,\delta} 
    \sum_{\varphi\in \calB_G,\nu_\varphi\le \varepsilon^{-1-\delta}} p^{-kn(n-1)}|\lambda_{\varphi}(p^{nk})|^2  |\varphi(x_0)|^2 
    \\ 
    +\sum_{P\ne G} C_P \sum_{\varphi\in \calB_P,\nu_\varphi\le \varepsilon^{-1-\delta}} p^{kn(2\theta_{\varphi,p}-(n-1)+\eta)} \intop_{\lambda \in i\a_P^*, \|\lambda\| \le \varepsilon^{-1-\delta}}  |\Eis_P(\varphi,\lambda)(x_0)|^2 \d \lambda 
    + \varepsilon^{N}
\end{multline*}
for every $N>0$.
\end{prop}

\begin{proof}
We notice that the proposition follows from \cref{eq:basic spectral decomposition} and \cref{useful-bound-h}, if we can show that the contribution of $\varphi\in\B_P$ with $\nu_\varphi \ge \varepsilon^{-1-\delta}$, and the contribution of $\varphi\in \B_p$ for $P\neq G$ and $\lambda$ with $\varphi \le \varepsilon^{-1-\delta}$ and $\|\lambda\|\ge \varepsilon^{-1-\delta}$ are $O_{\Omega,N,\delta}(\varepsilon^{N})$.

We use the estimate 
\[|\tilde{h}_{p^{nk}}(\mu_{\varphi,\lambda,p)})| \le 1,\]
which follows from \cref{rem:trivial-bound-spherical-transform},
throughout the proof.

We first handle the contribution to \cref{eq:basic spectral decomposition} from $\varphi$ with $S\le \nu_\varphi \le 2S$, with $S\ge \varepsilon^{-1-\delta}$. Notice that in this case 
$(1+\varepsilon \nu_{\varphi}) \gg S^{\delta'}$
for some $\delta'$ depending on $\delta$.
Applying \cref{useful-bound-tilde-k} we see that the contribution of such $\varphi$ is bounded by 
\begin{align*}
\ll_{N,\delta} \sum_P \sum_{\varphi\in \calB_P,S\le \nu_\varphi\le 2S} S^{-N}\intop_{i\a_P^*} (1+\varepsilon\|\lambda\|)^{-N}|\Eis_P(\varphi,\lambda)(x_0)|^2\d\lambda.
\end{align*}
Take $N$ large enough, use \cref{prop:Weyl Law local}, and apply integration by parts. 
Then the inner integral is bounded by $\ll_\Omega \varepsilon^{-L_1}$
for some absolute $L_1$.
Making $N$ large enough the entire sum is bounded by 
\[
\ll_{\Omega,N,\delta} \sum_P \sum_{\varphi\in \calB_P,S\le \nu_\varphi\le 2S} S^{-N}.
\]
Using \cref{prop:Weyl law upper bound} and \cref{lem:Weyl law residual} we have
\[
|\{\varphi\in \calB_P,S\le \nu_\varphi\le 2S\}| \le S^{L_2}
\]
for some absolute $L_2$. 
Once again making $N$ sufficiently large the entire contribution is bounded by $\ll_{\Omega,N,\delta} S^{-N}$. Summing over $S\ge \varepsilon^{-1-\delta}$ in dyadic intervals we deduce the entire contribution from $\nu_{\varphi}\ge \varepsilon^{-1-\delta}$ is bounded by $\ll_{\Omega,N,\delta} \varepsilon^{N}$.

We next deal with the case when $\nu_\varphi \le \varepsilon^{-1-\delta}$ and $\|\lambda \|\ge \varepsilon^{-1-\delta}$. 
Applying \cref{useful-bound-tilde-k} we see that the contribution is bounded by 
\begin{align*}
\ll_N \sum_P \sum_{\varphi\in \calB_P,\nu_\varphi\le \varepsilon^{-1-\delta}} (1+\varepsilon\nu_\varphi)^{-N} \intop_{\lambda \in i\a_P^*, \|\lambda\| \ge \varepsilon^{-1-\delta}}(1+\varepsilon\|\lambda\|)^{-N} |\Eis_P(\varphi,\lambda)(x_0)|^2\d\lambda.
\end{align*}
Using \cref{prop:Weyl Law local}, applying integration by parts, 
and making $N$ large enough, the inner integral is bounded by $\ll_{\Omega,N,\delta} \varepsilon^N$.
Therefore, the entire sum is bounded by 
\[
\ll_{\Omega,N,\delta} \varepsilon^N \sum_P \sum_{\varphi\in \calB_P,\nu_\varphi\le \varepsilon^{-1-\delta}} 1.
\]
Applying \cref{prop:Weyl law upper bound} and \cref{lem:Weyl law residual} again, and making $N$ sufficiently large, this is bounded by $\ll_{\Omega,N,\delta} \varepsilon^{N}$.
\end{proof}

A similar proposition treats the averaged version.
\begin{prop}\label{prop:truncation global}
Let $\Omega\subset \X_0$ be a fixed compact set. For every $\delta,\eta>0$ we have
\begin{align*}
    \intop_\Omega \|T^*(p^{nk})K_{\varepsilon,x_0}^{\X_0}\|_2^2 \d x_0 &\ll_{\Omega,N,\eta,\delta} \sum_{\varphi\in \calB_G,\nu_\varphi\le \varepsilon^{-1-\delta}} p^{-kn(n-1)}|\lambda_{\varphi}(p^{nk})|^2 \\
    \\ 
    & + \sum_{P\ne G} C_P \sum_{\varphi\in \calB_P,\nu_\varphi\le \varepsilon^{-1-\delta}} \varepsilon^{-(1+\delta +\eta)\dim \a^*_P} p^{kn(2\theta_{\varphi,p}-(n-1)+\eta)} + \varepsilon^N,
\end{align*}
for every $N>0$.
\end{prop}

\begin{proof}
We integrate both sides of the estimate in \cref{prop:truncation local} over $x_0\in\Omega$.
Apply \cref{thm:maass-selberg} to bound
\[
\intop_{\Omega}
\intop_{\lambda \in i\a_P^*, \|\lambda\| \le \varepsilon^{-1-\delta}} |\Eis_P(\varphi,\lambda)(x_0)|^2 \d \lambda \d x_0 \ll_{\Omega,\eta} \varepsilon^{-(1+\delta)(1+\eta)\dim a^*_P},
\]
for $\nu_\varphi\le\varepsilon^{-1-\delta}$. The claim follows from the fact that $\varphi\in\B_G$ are $L^2$-normalized.
\end{proof}

\section{Proof of \texorpdfstring{\cref{thm:kappa theorem intro}}{Theroem~2} and \texorpdfstring{\cref{thm:optimal-SDH}}{Theorem~3} }
\label{sec:proof of main theorem}

The goal of this section is to prove \cref{thm:kappa theorem intro} and \cref{thm:optimal-SDH}. By \cref{lem:kappa from L2 global convex}, to prove that $\kappa \le \beta:=\frac{n-1}{n-1-2\theta_n}$ it is sufficient to prove that for $\varepsilon_0$ small enough, for every $\varepsilon<\varepsilon_0$ and for 
$k = \lfloor \beta\frac{n+2}{2n}\log_p(\varepsilon^{-1}) \rfloor$
\[
\intop_\Omega\|T^*(p^{nk})K_{\varepsilon,x_0}^{\X_0})\|_2 \d x_0 \ll_\eta \varepsilon^{-\eta},
\]
for every $\eta>0$.

Using \cref{prop:truncation global} and standard modifications, it is sufficient to prove that under the same conditions that
\begin{equation}\label{eq:discrete estimate}
p^{-kn(n-1)}\sum_{\varphi\in \calB_G,\nu_\varphi\le \varepsilon^{-1}} |\lambda_{\varphi}(p^{nk})|^2 \ll_\eta \varepsilon^{-\eta},    
\end{equation}
and for every standard parabolic $P\ne G$
\begin{equation}\label{eq:continuous estimate}
\sum_{\varphi \in \calB_P, \nu_\varphi \le \varepsilon^{-1}} \varepsilon^{-\dim \a_P} p^{kn(2\theta_{\varphi, p}-(n-1)}) \ll_\eta \varepsilon^{-\eta},    
\end{equation}
for every $\eta>0$.

\subsection{The discrete spectrum}\label{subsec:pf- discrete spectrum}
We start by handling \cref{eq:discrete estimate}. We can further divide $\calB_G$ according to shapes, as in \cref{subsec:residual sepctrum}. Each such shape is of the form $S=((a,b))$, for $n=ab$.
We can uniformly bound 
\[
\theta_{\varphi,p} \le (b-1)/2+\theta_a,
\]
where $\theta_a$ is the best known bound towards the GRC for $\GL_a$.
So by \cref{lem:bounds on lambda from theta} we have
\[
|\lambda_{\varphi}(p^{nk})| \ll_\eta p^{nk((b-1)/2+\theta_a+\eta)}.
\]
We have $\theta_a =0$ for $a=1$, and by \cref{eq:luo-rudnick-sarnak} $\theta_a \le \frac{1}{2}-\frac{1}{a^2+1}$.
By \cref{lem:Weyl law residual}, we have
\[
\#\{\varphi\in \calB_S,\nu_\varphi\le \varepsilon^{-1}\} \ll \varepsilon^{-(a+2)(a-1)/2}.
\]
Therefore, applying the above bounds we get
\[
p^{-kn(n-1)}\sum_{\varphi\in \calB_S,\nu_\varphi\le \varepsilon^{-1}} |\lambda_{\varphi}(p^{nk})|^2 \ll \varepsilon^{-(a+2)(a-1)/2}p^{nk (b-1-(n-1)+2\theta_a+\eta)}.
\]
plugging in $k = \lfloor \beta\frac{n+2}{2n}\log_p(\varepsilon^{-1}) \rfloor$, the last value is
\[
\asymp  \varepsilon^{-\eta}\varepsilon^{-(a+2)(a-1)/2-\beta (n+2)(b-1-(n-1)+2\theta_a)/2}.
\]
Making $\eta$ small enough it suffices to show that
\begin{equation}\label{eq: equation ab}
\beta(n+2)(n-1-(b-1)-2\theta_a)-(a+2)(a-1) \ge 0.    
\end{equation}
For $a=1$, $b=n$ we have $\theta_a=0$. Hence \cref{eq: equation ab} is obvious for any $\beta\ge 1$. 

For $1<a<n$, $b=n/a$ we have $2\theta_a \le 1$. Then $n+2\ge a+2$ and \[n-b-2\theta_a> n/2-1\ge a-1.\] So \cref{eq: equation ab} still holds for any $\beta \ge 1$.

Finally, for $a=n,b=1$, \cref{eq: equation ab} will hold for as long as
\[
\beta \ge \frac{n-1}{n-1-2\theta_n}.
\]

Now, assuming \cref{conj:density conj} for $n\ge 4$ and using \cref{subconvex-density-2} and \cref{subconvex-density-3} for $n=2$ and $n=3$, respectively we can handle $a=n$, $b=1$ for $\beta=1$. Indeed, in this case $\calB_S = \calB_{n,\cusp}$. We have
\[
p^{-kn(n-1)}\sum_{\varphi\in \calB_S,\nu_\varphi\le \varepsilon^{-1}} |\lambda_{\varphi}(p^{nk})|^2 
\ll_\eta p^{-kn(n-1)}(\varepsilon^{-1}p^{nk})^\eta(\varepsilon^{-d}+p^{nk(n-1)}).
\]
Assuming $\beta=1$, i.e., $k = \lfloor \frac{n+2}{2n}\log_p(\varepsilon^{-1})\rfloor$, then 
\[
p^{nk(n-1)} \asymp \varepsilon^{(n+2)(n-1)/2}=\varepsilon^{-d},
\]
so the above is $\ll_\eta \varepsilon^{-\eta}$, as needed.

\subsection{The continuous spectrum}\label{subsec:pf- cont spectrum}
In this subsection we prove \cref{eq:continuous estimate} for every $\beta\ge 1$, which is enough for our purpose.

We further divide into shapes, as in \cref{subsec:residual sepctrum}. Let $S= ((a_1,b_1),\dots .,(a_r,b_r))$ be a shape, and let $\calB_S \subset \calB_P$. Notice that in this case $\dim \a_P = r-1$. Since we assume that $P\ne G$, we assume that $r>1$.

We need to prove that for every shape $S$, it holds that 
\[
\sum_{\varphi \in \calB_S, \nu_\varphi \le \varepsilon^{-1}} \varepsilon^{-(r-1)} p^{kn(2\theta_{\varphi, p}-(n-1))} \ll_\eta \varepsilon^{-\eta}.
\]
Without loss of generality, we assume that it holds that \[\max_i\{(b_i-1)/2+\theta_{a_i}\}=(b_1-1)/2+\theta_{a_1}.\]
Then by \cref{lem:residual theta bound} we have
$\theta_{\varphi, p} \le (b_1-1)/2+\theta_{a_1}$. Using \cref{lem:Weyl law residual}, we deduce
\[
\sum_{\varphi \in \calB_S, \nu_\varphi \le \varepsilon^{-1}} \varepsilon^{-(r-1)} p^{kn(2\theta_{\varphi, p}-(n-1))}
\ll \varepsilon^{-(r-1)-\sum_{i=1}^r(a_i+2)(a_i-1)/2}p^{kn(b_1-1+2\theta_{a_1}-(n-1))}.
\]
We plug in $k = \lfloor \frac{n+2}{2n}\log_p(\varepsilon^{-1})\rfloor$, and deduce that it is sufficient to prove that for every shape $S$ it holds that
\[
    -(r-1)-\sum_{i=1}^r(a_i+2)(a_i-1)/2+\frac{n+2}{2}((n-1)-(b_1-1)-2\theta_{a_1}) \ge 0.
\]
we start by noticing that 
\[
r+\sum_{i=1}^r(a_i+2)(a_i-1)/2 = \sum_{i=1}^r((a_i+2)(a_i-1)/2+1)=\sum_{i=1}^r a_i(a_i+1)/2.
\]
we have the following simple lemma
\begin{lemma}
Assume that $\sum_{i=2}^r a_i \le M$.
Then \[\sum_{i=2}^r a_i(a_i+1)/2 \le M(M+1)/2.\]
\end{lemma}
\begin{proof}
The polynomial $p(x)=x(x+1)/2$ satisfies for $x_1,x_2\ge0$ that $p(x_1)+p(x_2)\le p(x_1+x_2)$. The lemma follows.
\end{proof}

In our case, we have $\sum_{i=2}^r a_i \le n-a_1b_1$. So we deduce that 
\[
\sum_{i=1}^r a_i(a_i+1)/2 \le (n-a_1b_1)(n-a_1b_1+1)/2.
\]
we deduce that it is sufficient to prove the following.
\begin{lemma}
Denote 
\[F(a,b,n) := 2+(n+2)(n-b-2\theta_{a})-a(a+1)-(n-ab)(n-ab+1).
\]
Then for every positive integers $n\ge 2$ and $a,b$ such that $ab<n$ it holds that
$F(a,b,n)\ge 0.$
\end{lemma}
\begin{proof}
We show by case-by-case analysis. First, it is easy to see that the claim holds for $n=2$, since then $a=b=1$ and $\theta_a=0$.

Now, assume that $a = 1$. Then $\theta_a=0$, and it holds that \[
F(1,b,n)=2+(n+2)(n-b)-2-(n-b)(n-b+1) = (n-b)(n+2-(n-b+1))\ge 0.
\]

Next, assume that $a>1$. We first take care of the case $n=3$. Then we only need to consider $a=2,b=1$ case. It holds that
\[
F(2,1,3)=2+5(2-2\theta_2)-6-2 = 4-10\theta_2.
\]
Plugging in the Kim--Sarnak's bound $\theta_2 \le \frac{7}{64}$ we get the desired result.

Now assume that $a>1$ and $n\ge4$. We will use the bound $\theta_a \le 1/2$ which follows from \cref{eq:luo-rudnick-sarnak}. Then 
\[
F(a,b,n)\ge G(a,b,n) := 2+(n+2)(n-b-1)-a(a+1)-(n-ab)(n-ab+1).
\]
First consider the case $b=1$. Then 
\[
G(a,1,n)= 2+(n+2)(n-2)-a(a+1)-(n-a)(n-a+1).
\]
The values of $G(a,1,n)$, when $n$ is fixed and $a$ varies, lie on a parabola with negative leading coefficient. To prove lower bounds in the range $2\le a\le n-1$ it is sufficient to check the extreme values, namely $a=1$ and $a=n-1$. 

We see that
\[
G(n-1,1,n)=G(1,1,n)= 2+(n+2)(n-2)-2-(n-1)n = n^2-4-n^2+n=n-4\ge0,
\]
as we assumed that $n\ge 4$.

Finally, we are left with the case $a\ge2,b\ge 2$. In this case we have $n\ge ab+1 \ge 5$ and $a\le (n-1)/2$. Fix $a$ and $n$. Then $G(a,b,n)$ as a function of $b$ is again a parabola with negative leading coefficient. So it suffices to check the extreme cases $b=1$ and $b=(n-1)/a$ 
We already proved that $G(a,1,n)\ge 0$, so we are left to show that for $2\le a\le (n-1)/2$, \[G(a,(n-1)/a,n)\ge0.\]
Indeed, we get 
\[
G(a,(n-1)/a,n) = 2+(n+2)(n-(n-1)/a-1)-a(a+1)-2.
\]
Using $(n-1)/a\le (n-1)/2$ and $a\le (n-1)/2$, we get
\begin{align*}
G(a,(n-1)/a,n)&\ge (n+2)((n-1)/2-1)-(n-1)(n+1)/4 \\
&=\frac{2(n+2)(n-3)-(n^2-1)}{4}=\frac{n^2-2n-11}{4}.
\end{align*}
The last value is non-negative since $n\ge 5$.
\end{proof}

\section{Proof of \texorpdfstring{\cref{thm:local-kappa-bound}}{Theorem~4}}
\label{sec: local exponents n=3}

We start the proof for $n$ general, assuming the GRC for $\GL_m$ for all $m\le n$.
The proof is similar to the last section, but we use \cref{lem:kappa from L2 local convex} instead of \cref{lem:kappa from L2 global convex}. 
It is therefore sufficient to show that for every $x_0 \in \X_0$, for $\varepsilon<\varepsilon_0$, for $k=\lfloor\frac{n+2}{2n}\log_p(\varepsilon^{-1}) \rfloor$, it holds that \[
\|T^*(p^{nk})K_{\varepsilon,x_0}^{\X_0}\|_2\,\,\ll_\eta \varepsilon^{-\eta}.
\]
We choose a compact subset $\Omega$ which contains $x_0$.
Using \cref{prop:truncation local}, we reduce to proving that 
\begin{equation}\label{eq:n=3 pf 1}
\sum_P C_P\sum_{\varphi\in \calB_P, \nu_\varphi \le \varepsilon^{-1}} p^{kn(2\theta_{\varphi,p}-(n-1))}
\intop_{\lambda\in i\a_P^*, \|\lambda\|\le \varepsilon^{-1}} |\Eis_P(\varphi,\lambda)(x_0)|^2 \d \lambda \ll \varepsilon^{-\eta}.    
\end{equation}
We can then further divide the above sum into shapes as done in \cref{sec:proof of main theorem}. The basic observation is that for a shape $S=((a_1,b_1),\dots ,(a_r,b_r))$, if $b_i=1$ for all $i$ (i.e., the shape is \emph{cuspidal}), then by \cref{lem:residual theta bound} for every $\varphi\in \calB_S$ the GRC implies that $\theta_{\varphi,p}=0$. 
Therefore, for cuspidal $S$, assuming GRC,
\begin{align*}
&\sum_{\varphi \in \calB_S, \nu_\varphi \le \varepsilon^{-1}} p^{kn(2\theta_{\varphi,p}-(n-1))}
\intop_{\lambda\in i\a_P^*, \|\lambda\|\le \varepsilon^{-1}} |\Eis_P(\varphi,\lambda)(x_0)|^2 \d \lambda  \\
=& p^{-kn(n-1)}\sum_{\varphi \in \calB_S, \nu_\varphi \le \varepsilon^{-1}} 
\intop_{\lambda\in i\a_P^*, \|\lambda\|\le \varepsilon^{-1}} |\Eis_P(\varphi,\lambda)(x_0)|^2 \d \lambda
\end{align*}
Using \cref{prop:Weyl Law local}, the last value is 
\[
\ll_\Omega p^{-kn(n-1)}\varepsilon^{-d},
\]
and plugging in the value of $k$ the last value is $\ll 1$. 

We, therefore, deduce that assuming the GRC, we are only left with non-cuspidal shapes to deal with.

For $S=((1,n))$, the only representation $\varphi \in \calB_S$ is the constant $L^2$-normalized function $\varphi_0$. In this case $\theta_{\varphi_0,p}=(n-1)/2$, and its contribution is 
\[
p^{kn(2\theta_{\varphi_0,p}-(n-1))}\varphi_0(x_0) \ll 1.
\]

Now, let us fix $n=3$. In this case, the only non-cuspidal shapes are $S=((1,3))$ and $S=((1,2),(1,1))$.
Following the above, we are left with the shape $S=((1,2),(1,1))$. In this case there is exactly one element $\varphi_1 \in\calB_S$ with $\theta_{\varphi_1,p}=1/2$. 
We have the following result.
\begin{prop}\label{prop:max-degn-eis-bound}
For every $\lambda \in i\a_P^*$ and every $\eta>0$, it holds that 
\[
|\Eis_P(\varphi_1,\lambda)(x_0)| \ll_{\eta,x_0} \lambda^{3/4+\eta}.
\]
\end{prop}
The result follows from the functional equation of $\Eis_P(\varphi_1,\lambda)$, standard bounds of the Riemann $\xi$-function, and the Phragm\'en--Lindel\"of convexity principle. This is explained in \cite{blomer2020epstein}.
As a matter of fact, \cite[Theorem~1]{blomer2020epstein} proves a stronger and far deeper result where the exponent is $1/2$ instead of $3/4$. 

Now, using \cref{prop:max-degn-eis-bound}, we deduce that
\begin{align*}
&\sum_{\varphi \in \calB_S, \nu_\varphi \le \varepsilon^{-1}} p^{3k(2\theta_{\varphi,p}-2)}
\intop_{\lambda\in i\a_P^*, \|\lambda\|\le \varepsilon^{-1}} |\Eis_P(\varphi,\lambda)(x_0)|^2 \d \lambda \\
&\ll_{\eta,x_0} p^{-3k}\varepsilon^{-5/2-\eta}.
\end{align*}
In this case $p^k \asymp \varepsilon^{-5/6}$, so we conclude.

\begin{remark}
The argument above using the local $L^\infty$-bound of the maximal degenerate Eisenstein series extends to the shapes of the form $S=((1,n-1),(1,1))$ for any $n$. However, for $n=4$ we do not know how to handle shapes of the form $S=((1,2),(1,2))$ or $S=((1,2),(2,1))$ or $S=((2,2))$. In all cases, we need good uniform bounds for $\Eis_P(\varphi,\lambda)(x_0)$.
\end{remark} 

\begin{remark}
Without assuming the GRC we do not know, even for $n=2$, whether $\kappa(x_0)=1$ for all $x_0\in\X$. 
The problem reduces to the following local version of Sarnak's density conjecture, which we state for general $n$: for every $x_0\in \X$ and $l\ge 0$, $T\ge 1$
\[\sum_{\varphi\in\calF_T} \left|\lambda_\varphi(p^l)\right|^2 |\varphi(x_0)|^2
\ll_{\delta,p,x_0}\left(Tp^l\right)^{\delta}\left(T^{d}+p^{l(n-1)}\right),\]
for every $\delta>0$.
This version is open even for $n= 2$ but can be proven for some $x_0$ using different methods. This result will appear elsewhere.
\end{remark}

\bibliographystyle{acm}
\bibliography{./database}

\end{document}